\documentclass[reqno, 12pt]{amsart}
\pdfoutput=1

\def\ssign{\textsection\nobreak\hspace{1pt plus 0.3pt}}
\makeatletter
\let\origsection=\section 
\def\mysection{\@mystartsection{section}{1}\z@{.7\linespacing\@plus\linespacing}{.5\linespacing}{\normalfont\scshape\centering\ssign}}
\def\section{\@ifstar{\origsection*}{\mysection}}
\def\appendix{\par\c@section\z@ \c@subsection\z@
	\let\sectionname\appendixname
	\let\section=\origsection
	\def\thesection{\@Alph\c@section}} 
\def\@mystartsection#1#2#3#4#5#6{\if@noskipsec \leavevmode \fi
	\par \@tempskipa #4\relax
	\@afterindenttrue
	\ifdim \@tempskipa <\z@ \@tempskipa -\@tempskipa \@afterindentfalse\fi
	\if@nobreak \everypar{}\else
	\addpenalty\@secpenalty\addvspace\@tempskipa\fi
	\@dblarg{\@mysect{#1}{#2}{#3}{#4}{#5}{#6}}}
\def\@mysect#1#2#3#4#5#6[#7]#8{\edef\@toclevel{\ifnum#2=\@m 0\else\number#2\fi}\ifnum #2>\c@secnumdepth \let\@secnumber\@empty
	\else \@xp\let\@xp\@secnumber\csname the#1\endcsname\fi
	\@tempskipa #5\relax
	\ifnum #2>\c@secnumdepth
	\let\@svsec\@empty
	\else
	\refstepcounter{#1}\edef\@secnumpunct{\ifdim\@tempskipa>\z@ \@ifnotempty{#8}{\@nx\enspace}\else
		\@ifempty{#8}{.}{\@nx\enspace}\fi
	}\@ifempty{#8}{\ifnum #2=\tw@ \def\@secnumfont{\bfseries}\fi}{}\protected@edef\@svsec{\ifnum#2<\@m
		\@ifundefined{#1name}{}{\ignorespaces\csname #1name\endcsname\space
		}\fi
		\@seccntformat{#1}}\fi
	\ifdim \@tempskipa>\z@ \begingroup #6\relax
	\@hangfrom{\hskip #3\relax\@svsec}{\interlinepenalty\@M #8\par}\endgroup
	\ifnum#2>\@m \else \@tocwrite{#1}{#8}\fi
	\else
	\def\@svsechd{#6\hskip #3\@svsec
		\@ifnotempty{#8}{\ignorespaces#8\unskip
			\@addpunct.}\ifnum#2>\@m \else \@tocwrite{#1}{#8}\fi
	}\fi
	\global\@nobreaktrue
	\@xsect{#5}}
\makeatother

\usepackage{amsmath,amssymb,amsthm}
\usepackage{mathrsfs}
\usepackage{mathabx}\changenotsign
\usepackage{dsfont}
\usepackage{bbm}

\usepackage{xcolor}
\usepackage[backref=section]{hyperref}
\usepackage[ocgcolorlinks]{ocgx2}
\hypersetup{
	colorlinks=true,
	linkcolor={red!60!black},
	citecolor={green!60!black},
	urlcolor={blue!60!black},
}

\usepackage[open,openlevel=2,atend]{bookmark}

\usepackage[abbrev,msc-links,backrefs]{amsrefs}
\usepackage{doi}

\renewcommand{\PrintDOI}[1]{\doi{#1}}

\usepackage[T1]{fontenc}
\usepackage{lmodern}
\usepackage[babel]{microtype}
\usepackage[english]{babel}

\usepackage{geometry}
\numberwithin{equation}{section}
\numberwithin{figure}{section}

\usepackage{enumitem}

\def\alabel{\upshape({\itshape \alph*\,})}

\let\polishlcross=\l
\def\l{\ifmmode\ell\else\polishlcross\fi}

\let\emptyset=\varnothing
\let\setminus=\smallsetminus

\makeatletter
\def\moverlay{\mathpalette\mov@rlay}
\def\mov@rlay#1#2{\leavevmode\vtop{   \baselineskip\z@skip \lineskiplimit-\maxdimen
		\ialign{\hfil$\m@th#1##$\hfil\cr#2\crcr}}}
\newcommand{\charfusion}[3][\mathord]{
	#1{\ifx#1\mathop\vphantom{#2}\fi
		\mathpalette\mov@rlay{#2\cr#3}
	}
	\ifx#1\mathop\expandafter\displaylimits\fi}
\makeatother

\newcommand{\dcup}{\charfusion[\mathbin]{\cup}{\cdot}}

\DeclareFontFamily{U}  {MnSymbolC}{}
\DeclareSymbolFont{MnSyC}         {U}  {MnSymbolC}{m}{n}
\DeclareFontShape{U}{MnSymbolC}{m}{n}{
	<-6>  MnSymbolC5
	<6-7>  MnSymbolC6
	<7-8>  MnSymbolC7
	<8-9>  MnSymbolC8
	<9-10> MnSymbolC9
	<10-12> MnSymbolC10
	<12->   MnSymbolC12}{}
\DeclareMathSymbol{\powerset}{\mathord}{MnSyC}{180}

\usepackage{tikz}
\usetikzlibrary{calc,decorations.pathmorphing}
\usetikzlibrary{arrows,decorations.pathreplacing}
\pgfdeclarelayer{background}
\pgfdeclarelayer{foreground}
\pgfdeclarelayer{front}
\pgfsetlayers{background,main,foreground,front}

\usepackage{multicol}
\usepackage{subcaption}

\newcommand{\pedge}[9]{
	
	\ifx\relax#6\relax
	\def\qoffs{0pt}
	\else
	\def\qoffs{#6}
	\fi
	
	\def\phedge{
		($#1+#5!\qoffs!-90:#2-#5$) -- 
		($#2+#1!\qoffs!-90:#3-#1$) -- 
		($#3+#2!\qoffs!-90:#4-#2$) -- 
		($#4+#3!\qoffs!-90:#5-#3$) -- 
		($#5+#4!\qoffs!-90:#1-#4$) -- cycle}

	\coordinate (12) at ($#1!\qoffs!90:#2$);
	\coordinate (15) at ($#1!\qoffs!-90:#5$);
	\coordinate (23) at ($#2!\qoffs!90:#3$);
	\coordinate (21) at ($#2!\qoffs!-90:#1$);
	\coordinate (34) at ($#3!\qoffs!90:#4$);
	\coordinate (32) at ($#3!\qoffs!-90:#2$);
	\coordinate (45) at ($#4!\qoffs!90:#5$);
	\coordinate (43) at ($#4!\qoffs!-90:#3$);
	\coordinate (51) at ($#5!\qoffs!90:#1$);
	\coordinate (54) at ($#5!\qoffs!-90:#4$);

	\def\nphedge{
		(15) let \p1=($(15)-#1$), \p2=($(12)-#1$) in 
		arc[start angle={atan2(\y1,\x1)}, delta angle={atan2(\y2,\x2)-atan2(\y1,\x1)-360*(atan2(\y2,\x2)-atan2(\y1,\x1)>0)}, x radius=\qoffs, y radius=\qoffs] --
		(21) let \p1=($(21)-#2$), \p2=($(23)-#2$) in 
		arc[start angle={atan2(\y1,\x1)}, delta angle={atan2(\y2,\x2)-atan2(\y1,\x1)-360*(atan2(\y2,\x2)-atan2(\y1,\x1)>0)}, x radius=\qoffs, y radius=\qoffs] --
		(32) let \p1=($(32)-#3$), \p2=($(34)-#3$) in 
		arc[start angle={atan2(\y1,\x1)}, delta angle={atan2(\y2,\x2)-atan2(\y1,\x1)-360*(atan2(\y2,\x2)-atan2(\y1,\x1)>0)}, x radius=\qoffs, y radius=\qoffs] --
		(43) let \p1=($(43)-#4$), \p2=($(45)-#4$) in 
		arc[start angle={atan2(\y1,\x1)}, delta angle={atan2(\y2,\x2)-atan2(\y1,\x1)-360*(atan2(\y2,\x2)-atan2(\y1,\x1)>0)}, x radius=\qoffs, y radius=\qoffs] --
		(54) let \p1=($(54)-#5$), \p2=($(51)-#5$) in 
		arc[start angle={atan2(\y1,\x1)}, delta angle={atan2(\y2,\x2)-atan2(\y1,\x1)-360*(atan2(\y2,\x2)-atan2(\y1,\x1)>0)}, x radius=\qoffs, y radius=\qoffs] --
		cycle}
	
	\ifx\relax#7\relax
	\def\plwidth{1pt}
	\else
	\def\plwidth{#7}
	\fi
	
	\ifx\relax#9\relax
	\fill \nphedge;
	\else
	\fill[#9]\nphedge;
	\fi
	
	\ifx\relax#8\relax
	\draw[line width=\plwidth,rounded corners=\qoffs]\nphedge;
	\else
	\draw[line width=\plwidth,#8]\nphedge;
	\fi
}

\newcommand{\qedge}[7]{
	
	\ifx\relax#4\relax
	\def\qoffs{0pt}
	\else
	\def\qoffs{#4}
	\fi
	
	\def\qhedge{
		($#1+#3!\qoffs!-90:#2-#3$) --
		($#2+#1!\qoffs!-90:#3-#1$) --
		($#3+#2!\qoffs!-90:#1-#2$) -- cycle}

	\coordinate (12) at ($#1!\qoffs!90:#2$);
	\coordinate (13) at ($#1!\qoffs!-90:#3$);
	\coordinate (23) at ($#2!\qoffs!90:#3$);
	\coordinate (21) at ($#2!\qoffs!-90:#1$);
	\coordinate (31) at ($#3!\qoffs!90:#1$);
	\coordinate (32) at ($#3!\qoffs!-90:#2$);
	
	\def\nqhedge{
		(13) let \p1=($(13)-#1$), \p2=($(12)-#1$) in
		arc[start angle={atan2(\y1,\x1)}, delta angle={atan2(\y2,\x2)-atan2(\y1,\x1)-360*(atan2(\y2,\x2)-atan2(\y1,\x1)>0)}, x radius=\qoffs, y radius=\qoffs] --
		(21) let \p1=($(21)-#2$), \p2=($(23)-#2$) in
		arc[start angle={atan2(\y1,\x1)}, delta angle={atan2(\y2,\x2)-atan2(\y1,\x1)-360*(atan2(\y2,\x2)-atan2(\y1,\x1)>0)}, x radius=\qoffs, y radius=\qoffs] --
		(32) let \p1=($(32)-#3$), \p2=($(31)-#3$) in
		arc[start angle={atan2(\y1,\x1)}, delta angle={atan2(\y2,\x2)-atan2(\y1,\x1)-360*(atan2(\y2,\x2)-atan2(\y1,\x1)>0)}, x radius=\qoffs, y radius=\qoffs] --
		cycle}
	
	\ifx\relax#5\relax
	\def\qlwidth{1pt}
	\else
	\def\qlwidth{#5}
	\fi
	
	\ifx\relax#7\relax
	\fill \nqhedge;
	\else
	\fill[#7]\nqhedge;
	\fi
	
	\ifx\relax#6\relax
	\draw[line width=\qlwidth,rounded corners=\qoffs]\nqhedge;
	\else
	\draw[line width=\qlwidth,#6]\nqhedge;
	\fi
}

\newcommand{\redge}[8]{
	
	\ifx\relax#5\relax
	\def\qoffs{0pt}
	\else
	\def\qoffs{#5}
	\fi
	
	\def\rhedge{
		($#1+#4!\qoffs!-90:#2-#4$) -- 
		($#2+#1!\qoffs!-90:#3-#1$) -- 
		($#3+#2!\qoffs!-90:#4-#2$) -- 
		($#4+#3!\qoffs!-90:#1-#3$) -- cycle}

	\coordinate (12) at ($#1!\qoffs!90:#2$);
	\coordinate (14) at ($#1!\qoffs!-90:#4$);
	\coordinate (23) at ($#2!\qoffs!90:#3$);
	\coordinate (21) at ($#2!\qoffs!-90:#1$);
	\coordinate (34) at ($#3!\qoffs!90:#4$);
	\coordinate (32) at ($#3!\qoffs!-90:#2$);
	\coordinate (41) at ($#4!\qoffs!90:#1$);
	\coordinate (43) at ($#4!\qoffs!-90:#3$);
	
	\def\nrhedge{
		(14) let \p1=($(14)-#1$), \p2=($(12)-#1$) in 
		arc[start angle={atan2(\y1,\x1)}, delta angle={atan2(\y2,\x2)-atan2(\y1,\x1)-360*(atan2(\y2,\x2)-atan2(\y1,\x1)>0)}, x radius=\qoffs, y radius=\qoffs] --
		(21) let \p1=($(21)-#2$), \p2=($(23)-#2$) in 
		arc[start angle={atan2(\y1,\x1)}, delta angle={atan2(\y2,\x2)-atan2(\y1,\x1)-360*(atan2(\y2,\x2)-atan2(\y1,\x1)>0)}, x radius=\qoffs, y radius=\qoffs] --
		(32) let \p1=($(32)-#3$), \p2=($(34)-#3$) in 
		arc[start angle={atan2(\y1,\x1)}, delta angle={atan2(\y2,\x2)-atan2(\y1,\x1)-360*(atan2(\y2,\x2)-atan2(\y1,\x1)>0)}, x radius=\qoffs, y radius=\qoffs] --
		(43) let \p1=($(43)-#4$), \p2=($(41)-#4$) in 
		arc[start angle={atan2(\y1,\x1)}, delta angle={atan2(\y2,\x2)-atan2(\y1,\x1)-360*(atan2(\y2,\x2)-atan2(\y1,\x1)>0)}, x radius=\qoffs, y radius=\qoffs] --
		cycle}
	
	\ifx\relax#6\relax
	\def\rlwidth{1pt}
	\else
	\def\rlwidth{#6}
	\fi
	
	\ifx\relax#8\relax
	\fill \nrhedge;
	\else
	\fill[#8]\nrhedge;
	\fi
	
	\ifx\relax#7\relax
	\draw[line width=\rlwidth,rounded corners=\qoffs]\nrhedge;
	\else
	\draw[line width=\rlwidth,#7]\nrhedge;
	\fi
}

\let\epsilon=\varepsilon
\let\eps=\epsilon
\let\rho=\varrho
\let\theta=\vartheta

\def\NN{{\mathds N}}

\newcommand{\cA}{\mathcal{A}}
\newcommand{\cB}{\mathcal{B}}
\newcommand{\cC}{\mathcal{C}}

\newcommand{\cE}{\mathcal{E}}
\newcommand{\cF}{\mathcal{F}}

\newcommand{\cQ}{\mathcal{Q}}
\newcommand{\cR}{\mathcal{R}}
\newcommand{\cS}{\mathcal{S}}
\newcommand{\cT}{\mathcal{T}}

\newcommand{\gM}{\mathfrak{M}}

\newcommand{\fB}{\mathfrak{B}}

\newtheoremstyle{note}  {4pt}  {4pt}  {\sl}  {}  {\bfseries}  {.}  {.5em}          {}
\newtheoremstyle{introthms}  {3pt}  {3pt}  {\itshape}  {}  {\bfseries}  {.}  {.5em}          {\thmnote{#3}}
\newtheoremstyle{remark}  {2pt}  {2pt}  {\rm}  {}  {\bfseries}  {.}  {.3em}          {}

\theoremstyle{plain}
\newtheorem{theorem}{Theorem}[section]
\newtheorem{lemma}[theorem]{Lemma}

\newtheorem{constr}{Construction}

\newtheorem{cor}[theorem]{Corollary}
\newtheorem{fact}[theorem]{Fact}
\newtheorem{claim}[theorem]{Claim}

\theoremstyle{note}
\newtheorem{dfn}[theorem]{Definition}

\theoremstyle{remark}
\newtheorem{remark}[theorem]{Remark}

\newtheorem{exmpl}[theorem]{Example}

\usepackage{lineno}
\newcommand*\patchAmsMathEnvironmentForLineno[1]{
	\expandafter\let\csname old#1\expandafter\endcsname\csname #1\endcsname
	\expandafter\let\csname oldend#1\expandafter\endcsname\csname end#1\endcsname
	\renewenvironment{#1}
	{\linenomath\csname old#1\endcsname}
	{\csname oldend#1\endcsname\endlinenomath}}
\newcommand*\patchBothAmsMathEnvironmentsForLineno[1]{
	\patchAmsMathEnvironmentForLineno{#1}
	\patchAmsMathEnvironmentForLineno{#1*}}
\AtBeginDocument{
	\patchBothAmsMathEnvironmentsForLineno{equation}
	\patchBothAmsMathEnvironmentsForLineno{align}
	\patchBothAmsMathEnvironmentsForLineno{flalign}
	\patchBothAmsMathEnvironmentsForLineno{alignat}
	\patchBothAmsMathEnvironmentsForLineno{gather}
	\patchBothAmsMathEnvironmentsForLineno{multline}
}

\usepackage{scalerel}

\makeatletter
\newcommand{\overrighharpoonup}[1]{\ThisStyle{		\vbox {\m@th\ialign{##\crcr
				\rightharpoonupfill \crcr
				\noalign{\kern-\p@\nointerlineskip}
				$\hfil\SavedStyle#1\hfil$\crcr}}}}

\def\rightharpoonupfill{	$\SavedStyle\m@th\mkern+0.8mu\cleaders\hbox{$\shortbar\mkern-4mu$}\hfill\rightharpoonuptip\mkern+0.8mu$}

\def\rightharpoonuptip{	\raisebox{\z@}[2pt][1pt]{\scalebox{0.55}{$\SavedStyle\rightharpoonup$}}}

\def\shortbar{	\smash{\scalebox{0.55}{$\SavedStyle\relbar$}}}
\makeatother

\makeatletter
\newcommand{\overlefharpoonup}[1]{\ThisStyle{		\vbox {\m@th\ialign{##\crcr
				\leftharpoonupfill \crcr
				\noalign{\kern-\p@\nointerlineskip}
				$\hfil\SavedStyle#1\hfil$\crcr}}}}

\def\leftharpoonupfill{	$\SavedStyle\m@th\mkern+0.8mu\cleaders\hbox{$\shortbar\mkern-4mu$}\hfill\leftharpoonuptip\mkern+0.8mu$}

\def\leftharpoonuptip{	\raisebox{\z@}[2pt][1pt]{\scalebox{0.55}{$\SavedStyle\leftharpoonup$}}}

\makeatletter
\newsavebox\myboxA
\newsavebox\myboxB
\newlength\mylenA

\newcommand*\xoverline[2][0.75]{	\sbox{\myboxA}{$\m@th#2$}	\setbox\myboxB\null	\ht\myboxB=\ht\myboxA	\dp\myboxB=\dp\myboxA	\wd\myboxB=#1\wd\myboxA	\sbox\myboxB{$\m@th\overline{\copy\myboxB}$}	\setlength\mylenA{\the\wd\myboxA}	\addtolength\mylenA{-\the\wd\myboxB}	\ifdim\wd\myboxB<\wd\myboxA	\rlap{\hskip 0.5\mylenA\usebox\myboxB}{\usebox\myboxA}	\else
	\hskip -0.5\mylenA\rlap{\usebox\myboxA}{\hskip 0.5\mylenA\usebox\myboxB}	\fi}
\makeatother

\let\vn=\varnothing
\let\lra=\longrightarrow
\let\sm=\smallsetminus
\def\ds{d_{\star}}

\DeclareSymbolFont{symbolsC}{U}{txsyc}{m}{n}
\SetSymbolFont{symbolsC}{bold}{U}{txsyc}{bx}{n}
\DeclareFontSubstitution{U}{txsyc}{m}{n}
\DeclareMathSymbol{\strictif}{\mathrel}{symbolsC}{74}

\begin{document}
	\title[Minimum degree in simplicial complexes]
	{Minimum degree in simplicial complexes}
	
	\author[Chr.~Reiher]{Christian Reiher}
	\address{Fachbereich Mathematik, Universit\"at Hamburg, Hamburg, Germany}
	\email{Christian.Reiher@uni-hamburg.de}
	
	\author[B. Sch\"ulke]{Bjarne Sch\"ulke}
	\address{Extremal Combinatorics and Probability Group, 
		Institute for Basic Science, Daejeon, South Korea}
	\email{schuelke@ibs.re.kr}
	\thanks{The second author was partially supported by the Young Scientist 
		Fellowship IBS-R029-Y7.}
	
	
	\begin{abstract}
		Given~$d\in\NN$, let~$\alpha(d)$ be the largest real number such 
		that every abstract simplicial complex~$\cS$ 
		with~$0<\vert\cS\vert\leq\alpha(d)\vert V(\cS)\vert$ has 
		a vertex of degree at most~$d$.
		We extend previous results by Frankl, Frankl and Watanabe, and Piga and 
		Sch\"ulke by proving that for all integers~$d$ and~$m$ with~$d\geq m\geq 1$, 
		we have~$\alpha(2^d-m)=\frac{2^{d+1}-m}{d+1}$.
		Similar results were obtained independently by Li, Ma, and Rong.
	\end{abstract}
	
	\maketitle
	
	\section{Introduction}\label{sec:intro}
	A (finite, abstract) {\it simplicial complex} is a non-empty finite 
	set~$\cS$ of finite sets that is closed under taking subsets, 
	meaning that for all~$S'\subseteq S\in\cS$, we have~$S'\in\cS$.
	According to this convention, the empty set $\vn$ is an element of every 
	simplicial complex. 
	We call~$V(\cS)=\bigcup_{S\in\cS}S$ the {\it vertex set} 
	of~$\cS$ and the members of $\cS$ {\it edges}. 
	The {\it degree}~$d(x)$ of a vertex~$x\in V(\cS)$ is the number of edges 
	containing~$x$, that is,~$d(x)=\vert\{e\in\cS\colon x\in e\}\vert$.
	The {\it minimum degree} of~$\cS$ is defined 
	as~$\delta(\cS)=\min\{d(x)\colon x\in V(\cS)\}$. The minimum 
	degree of the {\it trivial complex}~$\cS=\{\vn\}$ is $\delta(\cS)=\infty$.
	
	For~$d\in\NN_0$, let~$\alpha(d)$ be the largest real number such that 
	every simplicial complex $\cS$
	with~$\vert\cS\vert\leq\alpha(d)\vert V(\cS)\vert$ 
	satisfies~$\delta(\cS)\leq d$. 
	In other words, $\alpha(d)$ is the largest constant for which the lower 
	bound~$|\cS|>\alpha(d) |V(\cS)|$ can be inferred from~$\delta(\cS)>d$. 
	(For the existence of this maximum we refer to Section~\ref{sec:upperbounds}). 
	Currently, no explicit formula for $\alpha(d)$ is in sight. 
	In this work, we determine~$\alpha(2^d-m)$ for all integers~$d$ and~$m$ 
	with~$d\geq m\geq 1$. 
	As we shall see, this means that we solve this problem for the first natural 
	regime of parameters.
	
	Historically, the investigation of~$\alpha(d)$ began in the context of traces of 
	families of finite sets.
		Frankl~\cite{F:83} reduced a certain problem on traces to the minimum 
	degree problem for simplicial complexes and obtained the first general result, 
	namely that for all~$d\in\NN$,
	\begin{align}\label{eq:Franklresult}
		\alpha(2^d-1)=\frac{2^{d+1}-1}{d+1}\,.
	\end{align}
	In addition, Frankl and Watanabe~\cite{FW:94} proved that for all~$d\in\NN$,
	\begin{align}\label{eq:FranklWatanaberesult}
		\alpha(2^d-2)=\frac{2^{d+1}-2}{d+1}
	\end{align}
	and~$\alpha(2^d)=\frac{2^{d+1}-1}{d+1}+\frac{1}{2}$.
	Piga and Sch\"ulke~\cite{PS:21} extended these results to values further away from 
	powers of two by showing that for all~$d,c\in\NN$ with~$d\geq 4c$, 
	we have~$\alpha(2^d-c)=\frac{2^{d+1}-c}{d+1}$.
	We refer the reader to the introduction of~\cite{PS:21} and the references therein 
	for further early results. Our main contribution is the following.
	
	\begin{theorem}\label{thm:main}
		Let~$d$ and~$m$ be integers with~$d\geq m\geq1$.
		Then~$\alpha(2^d-m)=\frac{2^{d+1}-m}{d+1}$.
	\end{theorem}
	As we will see in Section~\ref{sec:upperbounds}, this formula is not correct 
	for~$m>d$. However, we illustrate that our methods can be used to push past this 
	natural barrier by showing the following result, which solves a conjecture by 
	Frankl and Watanabe~\cite{FW:94}.
	
	\begin{theorem}\label{thm:11}
		We have~$\alpha(11)=\frac{53}{10}$.
	\end{theorem}
	
	All other values of $\alpha(d)$ for $d\le 16$ have been known before 
	and for the reader's convenience we include a complete table. 
	
	\smallskip
	
	\begin{center}
		\begin{tabular}{c|c|c|c|c|c|c|c|c|c|c|c|c|c|c|c|c|c}
			$d=$ & $0$ & $1$ & $2$ & $3$ & $4$ & $5$ & $6$ & $7$ & $8$ & $9$ & $10$ & $11$ & $12$
			& $13$ & $14$ & $15$ & $16$ \\ \hline
			$\alpha(d)=$ & $1$ & $\frac 32$ & $2$ & $\frac 73$ & $\frac{17}6$ & $\frac{13}4$ 
			& $\frac 72$ & $\frac{15}4$ & $\frac{17}4$ & $\frac{65}{14}$ & $5$ & $\frac{53}{10}$
			& $\frac{28}5$ & $\frac{29}5$ & $6$ & $\frac{31}5$ & $\frac{67}{10}$
		\end{tabular} 
	\end{center}
	
	\smallskip
	
	Similar results were obtained independently by Li, Ma, and Rong~\cite{LMR:24}.
	Our methods are also applicable to cases where $d$ is slightly larger than a power 
	of two. For instance, we checked that $\alpha(17)=\frac{50}7$ and $\alpha(20)=8$, 
	and hope to return to this topic in the near future. 
	
	\subsection*{Overview}
	In Section~\ref{sec:upperbounds} we establish the existence of~$\alpha(d)$ and discuss upper bound constructions.
	The main part of this work concerns the lower bound:
	Given a simplicial complex with a certain minimum degree, we need to show that 
	it has many edges.
	Very roughly speaking, our proof follows the strategy developed in~\cite{PS:21}, 
	which has a local and a global part.
	Instead of counting the edges directly, it is convenient to instead consider weights 
	of vertices, which are terms similar to~$\sum_{x\in F\in\cS}\frac{1}{\vert F\vert}$.
	The minimum degree condition gives some lower bound for these weights.
	However, this lower bound is too small to finish directly.
	For this, we group vertices together into ``conglomerates'' (which are similar to 
	``clusters'' in~\cite{PS:21}) with the intention that the average weight in each
	conglomerate is sufficiently large.
	To bound the average weight in each conglomerate, we need local results about 
	simplicial complexes on a constant number of vertices with a certain number of edges.
	
	The improvements over~\cite{PS:21} come from several places.
	One is that the local considerations here are much more precise; we will discuss these 
	in Section~\ref{sec:local}.
	Another improvement comes from the fact that conglomerates are more flexible than 
	clusters---for instance, two conglomerates might intersect, albeit in at most one 
	vertex.
	Further, we introduce an auxiliary function on the vertices and a more general setting 
	with potentially several simplicial complexes on the same vertex set.
	Together, this allows us to perform an induction on the number of edges of one of the 
	simplicial complexes, which normally would disturb the minimum degree condition.
	The global part of our argument is given in Section~\ref{sec:mountains}.
	There we prove a more general result (Theorem~\ref{thm:meta}) and deduce 
	Theorem~\ref{thm:main} from it.
	In Section~\ref{sec:newreg} we prove Theorem~\ref{thm:11}.
	Section~\ref{sec:tools} serves as a collection of several tools, starting with inequalities, continuing with a weighted version of the Kruskal-Katona theorem, which is at the core of the local considerations in Section~\ref{sec:local}, and culminating in a first illustration of our method (see Lemma~\ref{lem:alpha=5gen}).
	
	\section{Upper bounds}\label{sec:upperbounds}
	
	Let us first check that for every $d\in\NN_0$ the maximum~$\alpha(d)$ does indeed exist. Given $d$ we consider the set 
	\[
	M=\{\alpha\in\mathds{R}\colon \text{Every simplicial complex $\cS$ with }
	\vert\cS\vert\leq\alpha\vert V(\cS)\vert
	\text{ satisfies }
	\delta(\cS)\leq d\}\,,
	\]
	which is bounded from above, as there are non-trivial complexes $\cS$ 
	with $\delta(\cS)>d$.
	Thus the supremum~$\alpha_\star=\sup(M)$ exists and it remains to 
	show $\alpha_\star\in M$. 
	Let~$\cS$ be any simplicial complex  
	with~$\vert\cS\vert\leq\alpha_\star\vert V(\cS)\vert$,
	take an isomorphic copy~$\cS'$ of~$\cS$ 
	with~${V(\cS)\cap V(\cS')=\vn}$, and observe 
	that $\vn\in\cS\cap\cS'$ 
	yields~$\vert\cS\cup\cS'\vert<\alpha_\star\vert V(\cS\cup\cS')\vert$. 
	Hence, there is some~$\alpha'\in M$ such 
	that~$\vert\cS\cup\cS'\vert\leq\alpha'\vert V(\cS\cup\cS')\vert$.
	Since~$\cS\cup\cS'$ is a simplicial complex, 
	the definition of~$M$ leads us to~$\delta(\cS)=\delta(\cS\cup\cS')\leq d$, 
	which proves~$\alpha_\star\in M$.
	
	It follows straight from the definition of~$\alpha(d)$ that every non-trivial simplicial 
	complex~$\cS$ with~$\delta(\cS)>d$ gives rise to an upper 
	bound $\alpha(d)\le|\cS|/|V(\cS)|$. We can again improve upon this by working 
	with many disjoint copies of~$\cS$. 
	
	\begin{lemma}\label{lem:constrgolbfromloc}
		If~$d\in\NN_0$ and~$\cS$ denotes a non-trivial simplicial complex 
		with~$\delta(\cS)>d$, 
		then
		\[
		\alpha(d)\leq\frac{\vert\cS\vert-1}{\vert V(\cS)\vert}\,.
		\]
	\end{lemma}
	
	\begin{proof}
		For any~$n\in\NN$, let~$\cS_1,\dots,\cS_n$ be pairwise 
		vertex-disjoint copies of~$\cS$. 
		Their union $\cS_\star=\cS_1\cup\dots\cup\cS_n$
		is a simplicial complex with~$\delta(\cS_\star)>d$, whence
		\[
		\alpha(d)
		\leq
		\frac{\vert\cS_\star\vert}{\vert V(\cS_\star)\vert}
		=
		\frac{n(\vert\cS\vert-1)+1}{n\vert V(\cS)\vert}\,.
		\]
		The result follows by letting~$n$ go to infinity.
	\end{proof}
	
	Now we can easily generate some (well-known) upper bounds on~$\alpha(\cdot)$.
	
	\begin{constr}\label{constr:base}
		For~$d\in\NN_0$, the power set~$\cS=\powerset([d+1])$ is a simplicial 
		complex with~$\vert\cS\vert=2^{d+1}$,~$\vert V(\cS)\vert=d+1$, 
		and~$\delta(\cS)=2^d$. 
		Thus, Lemma~\ref{lem:constrgolbfromloc} 
		yields~$\alpha(2^d-1)\leq\frac{2^{d+1}-1}{d+1}$.
	\end{constr}
	
	This can be generalised by omitting some `large' sets.
	
	\begin{constr}\label{constr:gen}
		For integers~$d$ and~$m$ with~$d\geq m\geq2$, the set system 
		\[
		\cS
		=
		\powerset([d+1])
		\setminus
		\Big(\{[d+1]\}\cup\big\{[d+1]\setminus\{i\}\colon i\in[m-2]\big\}\Big)
		\] 
		is a simplicial complex with~$\vert\cS\vert=2^{d+1}-m+1$,~$\vert V(\cS)\vert=d+1$, and~$\delta(\cS)=2^d-m+1$.
		Using Lemma~\ref{lem:constrgolbfromloc}, 
		we get~$\alpha(2^d-m)\leq\frac{2^{d+1}-m}{d+1}$.
	\end{constr}
	
	\begin{constr}\label{constr:beyond}
		For~$d\geq2$, the simplicial complex
		\begin{align*}
			\cS
			=&
			\{A\subseteq[d+1]\colon \vert A\vert\leq d-1\}\\
			&\cup\{A\subseteq[d+2,2d+2]\colon \vert A\vert\leq d-1\}\\
			&\cup\{\{1,d+2\},[2,d+1],[d+3,2d+2]\}
		\end{align*}
		satisfies~$\vert V(\cS)\vert=2(d+1)$,~$\vert\cS\vert=2(2^{d+1}-d-1)$, 
		and~$\delta(\cS)=2^d-d$, 
		whence
		\begin{equation}\label{eq:1551}
			\alpha(2^d-d-1)\leq\frac{2^{d+1}-d-3/2}{d+1}\,.
		\end{equation}
	\end{constr}
	
	This means that for~$m=d+1$, the formula in Theorem~\ref{thm:main} does not hold.
	Let us also note that we have already shown all upper bounds promised in 
	the introduction (as the case $d=4$ of~\eqref{eq:1551} 
	reads $\alpha(11)\le\frac{53}{10}$), so that it remains to establish the 
	corresponding lower bounds in the remainder of this article. 
	
	\section{Preliminaries}\label{sec:tools}
	
	\subsection{Inequalities}
	
	In this subsection we collect four estimates, which will be required later. 
	
	\begin{fact}\label{fa:ineq1}
		Given integers~$d\geq 3$ and~$t\ge 1$, let~$n_1,\dots,n_t\in\NN_0$ 
		be such that~$\sum_{\tau\in[t]}n_{\tau}\leq d$.
		If~$\sum_{\tau\in[t]}(2^{n_{\tau}}-1)\geq 2^d-d$,
		then there is an index~$\tau^\star\in[t]$ such that~$n_{\tau^\star}=d$ 
		and~$n_{\tau}=0$ for all~$\tau\in[t]\setminus\{\tau^\star\}$.
	\end{fact}
	
	\begin{proof}
		For all~$a,b\geq0$, we have~$(2^a-1)(2^b-1)\geq0$ and, 
		thereby,~$(2^a-1)+(2^b-1)\leq2^{a+b}-1$.
		By an induction on the number of summands this yields
		\[
		(2^{a_1}-1)+\dots+(2^{a_m}-1)
		\leq
		2^{a_1+\dots+a_m}-1
		\] 
		for all~$a_1,\dots,a_m\geq 0$.
		
		Now assume for the sake of contradiction that there is a 
		partition~$[t]=A\dcup B$ such that 
		both sums~$a=\sum_{i\in A}n_i$ and~$b=\sum_{i\in B}n_i$ are positive integers. 
		Using the condition in the statement and the above observation, we get
		\begin{align*}
			2^d-d
			&\leq
			\sum_{i\in[t]}(2^{n_i}-1)
			\leq(2^a-1)+(2^b-1) \\
			&=
			2\big((2^{a-1}-1)+(2^{b-1}-1)\big)+2
			\leq 2(2^{a+b-2}-1)+2\\
			&=
			2^{a+b-1}\leq 2^{d-1}\,.
		\end{align*}    
		But this gives~$d\geq 2^{d-1}$, contrary to~$d\geq 3$.
		
		We have thereby shown that there is some~$\tau^\star\in [t]$ with~$n_\tau=0$ 
		for all~$\tau\in [t]\sm\{\tau^\star\}$. 
		Finally,~$2^{d-1}<2^d-d\le 2^{n_{\tau^\star}}-1$ yields~$n_{\tau^\star}=d$.
	\end{proof}
	
	\begin{fact}\label{fa:ineq1+1/2}
		For all integers~$d\ge 4$, we have
		\[
		\frac{2^{d+1}-1}{2(d+1)}
		\ge
		\frac{1}{2}+\frac{2d-2}{3}+\frac{(d-2)(d-3)}{4}\,.
		\]
	\end{fact}
	
	\begin{proof}
		First note that 
		\[
		2^{d+1}
		\geq 
		1+\binom{d+1}{1}+\binom{d+1}{2}+\binom{d+1}{3}+\binom{d+1}{4}
		\]
		yields 
		\[
		\frac{2^{d+1}-1}{2(d+1)}
		\geq
		\frac{1}{2}+\frac{d}{4}+\frac{1}{6}\binom{d}{2}+\frac{1}{8}\binom{d}{3}\,.
		\]
		Together with
		\[
		\binom{d}{2}
		=
		\binom{d-2}{2}+2\binom{d-2}{1}+1
		\]
		and
		\[
		\binom{d}{3}
		=
		\binom{d-3}{3}+3\binom{d-3}{2}+3\binom{d-3}{1}+1
		\geq 
		3\binom{d-2}{2}+1
		\]
		this leads to
		\begin{align*}
			\frac{2^{d+1}-1}{2(d+1)}
			\geq& 
			\frac{1}{2}+\frac{d}{4}+\frac{1}{6}\binom{d-2}{2}+\frac{2d-3}{6}+\frac{3}{8}\binom{d-2}{2}+\frac{1}{8}\\
			=&
			\frac{1}{2}+\frac{1}{2}\binom{d-2}{2}+\frac{1}{24}\Big(\binom{d-2}{2}+14d-9\Big)\,.
		\end{align*}
		Again using
		\[
		\binom{d-2}{2}
		=
		\binom{d-4}{2}+2\binom{d-4}{1}+1\geq2(d-4)+1
		=
		2d-7
		\]
		leaves us with
		\[
		\frac{2^{d+1}-1}{2(d+1)}
		\geq
		\frac{1}{2}+\frac{(d-2)(d-3)}{4}+\frac{2d-2}{3}\,. \qedhere
		\]
	\end{proof}

\begin{fact}\label{fa:dmqrr's}
	If~$d$,~$m$, and~$s$ are integers with $d\ge m\ge 3$ and $m-2\ge s\ge 0$,
	then the numbers~$r'$,~$r$ defined by
	\begin{align*}
		r'=&\max\{a\in\NN_0\colon a\leq s\text{ and }2^a\leq m-1\}\,,\\
		r=&\max\{a\in\NN_0\colon a\leq s+1\text{ and }2^a\leq m\}
	\end{align*}
	satisfy 
	\[
	\frac{m-1-s}{d+1-r}-\frac12\le\frac{m-2-s}{d+1-r'}
	\]
	and 
	\[
	(s+1)\Big(\frac{s}{d}+\frac{m-2-s}{d+1-r'}\Big)
	+(d-s)\Big(\frac{s+1}{d}+\frac{m-2-s}{d+1-r}-\frac{1}{2}\Big)
	\leq 
	m-2\,.
	\]
\end{fact}

\begin{proof}
	Setting $x=m-2-s$ and $y=d+1-r$, we have $x\ge 0$ and 
	\[
	y-x=(d-m)+(s+1-r)+2\ge 2\,.
	\]
	This gives 
	\begin{align*}
		\frac{x+1}y
		&=
		\frac x{y+1}+\frac x{y(y+1)} +\frac 1y
		\le
		\frac x{y+1}+\frac x{2y} +\frac 1y \\
		&=
		\frac x{y+1}+\frac{x+2}{2y}
		\le
		\frac x{y+1}+\frac 12\,,
	\end{align*}
	and together with $r\le r'+1$ we obtain 
	\[
	\frac{m-1-s}{d+1-r}-\frac 12
	\le
	\frac{m-2-s}{d+2-r}
	\le	
	\frac{m-2-s}{d+1-r'}\,,
	\]
	which proves the first estimate. 
	
	Because of 
	\[
	\frac{(s+1)s}{d}+\frac{(d-s)(s+1)}{d}
			=s+1\,,
	\]
	the second one reduces to
	\[
	(m-2-s)\Big(\frac{s+1}{d+1-r'}+\frac{d-s}{d+1-r}\Big)
	\leq 
	m-s-3+\frac{d-s}{2}\,.
	\]
	Due to $d\ge m$, the right side is at least $\frac{3}{2}(m-2-s)$ 
	and, thus, it is sufficient to show
	\begin{align}\label{eq:dmqrr'suff}
		\frac{s+1}{d+1-r'}+\frac{d-s}{d+1-r}
		\leq
		\frac{3}{2}\,.
	\end{align}
		
	Suppose first that~$(d,r)=(4,2)$.
	Now~$2^r\leq m\leq d$ yields~$m=4$, while~$r\leq s+1$ and~$2^{r'}\leq m-1$ 
	give~$r'\leq 1\leq s$.
	Hence, we have indeed
	\[
	\frac{s+1}{d+1-r'}+\frac{d-s}{d+1-r}
	\leq
	\frac{s+1}{4}+\frac{4-s}{3}
	=
	\frac{19-s}{12}
	\leq
	\frac{18}{12}=\frac{3}{2}\,.
	\]
	
	Now let~$(d,r)\neq(4,2)$.
	Clearly,~$\frac{s+1}{d+1-r'}+\frac{d-s}{d+1-r}\leq\frac{d+1}{d+1-r}$, 
	and so to prove~\eqref{eq:dmqrr'suff} it is enough to 
	show~$r\leq\frac{d+1}{3}$, i.e.,~$d\geq 3r-1$.
	If~$r\neq 2$, then~$d\geq 2^r\geq 3r-1$.
	Otherwise,~$r=2$ holds, and so~$d\neq 4$ and~$d\geq2^r$ show~$d\geq 3r-1$.
\end{proof}

\begin{cor}\label{fa:ineq2}
	Let~$d$, $m$, $r$, and $s$ be nonnegative integers. If $d\ge m\ge 3$, 
	$m-1\ge s+1\ge r$, and $m\ge 2^r$, then 
	\[
	\frac{s+1}{d}+\frac{m-2-s}{d+1-r}-\frac{1}{2}\leq\frac{m-2}{d+1}\,.
	\]
\end{cor}

\begin{proof}
	For reasons of monotonicity, we can assume that $r$ is the number defined 
	in Fact~\ref{fa:dmqrr's}. In particular, we have $r\ge 1$ and 
	the first estimate from that result yields 
		\[
	\frac{s+1}{d}+\frac{m-2-s}{d+1-r}-\frac{1}{2}
	\le
	\frac{s}{d}+\frac{m-1-s}{d+1-r}-\frac{1}{2}
	\le 
	\frac{s}{d}+\frac{m-2-s}{d+1-r'}\,,
	\]
		for~$r'$ defined as in Fact~\ref{fa:dmqrr's}.
	Combined with the second part of Fact~\ref{fa:dmqrr's} this proves 
	our claim. 
\end{proof}

\subsection{Variants of Katona's theorem}

We shall often need a weighted version of the classical Kruskal-Katona 
theorem applicable to generalised complexes (see Lemma~\ref{lem:Katona} below). 
Let us first fix some notation. 
For two finite subsets $A$, $B$ of $\NN$, we write $A\strictif B$ if 
$\max(A \,\triangle\, B)\in B$, or, equivalently, if 
$\sum_{i\in A}2^{i-1}<\sum_{i\in B} 2^{i-1}$. The latter description of~$\strictif$
shows that the linear orders $\bigl(\bigcup_{m\ge 0}\NN^{(m)}, \strictif\bigr)$ 
and $(\NN, <)$ are isomorphic. 

\begin{dfn}\label{def:initfam}
Given~$\ell\in\NN$ and~$M\subseteq \NN_0$, let~$\cR_M(\ell)$ 
denote the set consisting of the first~$\ell$ elements 
of~$\bigl(\bigcup_{m\in M}\NN^{(m)},\strictif\bigr)$.
In the special case $M=\NN_0$ we drop the subscript and if~$m\in\NN_0$, 
we write~$\cR_m(\ell)$ instead of~$\cR_{\{m\}}(\ell)$.
\end{dfn}

Next, for a collection~$\mathcal{F}$ of finite sets and~$\ell\in\NN$, 
the~{\it $\ell$-shadow of~$\mathcal{F}$} is defined by
\[
\partial_{\ell}\mathcal{F}
=
\{F'\subseteq F\colon F\in\mathcal{F}\text{ and }\vert F'\vert=\ell\}\,.
\]
Let us recall that Kruskal~\cite{Kru63} and Katona~\cite{Ka68} proved independently 
that if $m\ge \ell$ are nonnegative integers and $\cF\subseteq \NN_0^{(m)}$ is 
finite, then $|\partial_\ell(\cF)|\ge |\partial_\ell(\cR_m(|\cF|))|$.
The following concept generalises simplicial complexes. 	

\begin{dfn}\label{def:Mcomplex}
Given~$M\subseteq\NN_0$, an~{\it $M$-complex} is a non-empty finite 
collection of sets~$\mathcal{K}$ such that
\begin{enumerate}
	\item $\vert x\vert\in M$ for all~$x\in\mathcal{K}$,
	\item and for all~$x\subseteq y\in\mathcal{K}$ with~$\vert x\vert\in M$, 
	we have~$x\in\mathcal{K}$.
\end{enumerate}
\end{dfn}

For instance,~$\NN_0$-complexes are the same as simplicial complexes.

\begin{remark}\label{rem:initfam}
If~$M\subseteq \NN_0$ 
and~$\cF\subseteq\bigcup_{m\in M}\NN^{(m)}$ is an~$M$-complex,
then, by the aforementioned Kruskal-Katona theorem, so 
is~$\bigcup_{m\in M}\cR_m(\vert\cF\cap\NN^{(m)}\vert)$. 
Indeed, for any~$\ell,m\in\NN$ with~$\ell\leq m$ this result 
yields~$\vert\partial_\ell\bigl(\cR_m(\vert\cF\cap\NN^{(m)}\vert)\bigr)\vert
\le \vert\partial_{\ell}(\cF\cap\NN^{(m)})\vert
\leq
\vert\cF\cap\NN^{(\ell)}\vert$.
As~$\partial_\ell\bigl(\cR_m(\vert\cF\cap\NN^{(m)}\vert)\bigr)$ is an initial 
segment of $(\NN^{(\ell)}, \strictif)$, it is thus a subset 
of~$\cR_\ell(\vert\cF\cap\NN^{(\ell)}\vert)$.
\end{remark}

The following lemma is essentially due to Katona~\cite{K:78}, but we will prove it 
for the sake of completeness.

\begin{lemma}\label{lem:Katona}
If~$M\subseteq \NN_0$ and~$\cF\subseteq\bigcup_{m\in M}\NN^{(m)}$ 
is an~$M$-complex, then for every family of reals~$(w_m)_{m\in M}$ 
satisfying~$w_m\geq w_{m'}$ whenever~$m\leq m'$, we have
\begin{align}\label{eq:weightedKatona}
	\sum_{m\in M}w_m\big\vert \cF\cap\NN^{(m)}\big\vert
	\geq
	\sum_{m\in M}w_m\big\vert \cR_M(\vert \cF\vert)\cap\NN^{(m)}\big\vert\,.
\end{align}
\end{lemma}

\begin{proof}
By Remark~\ref{rem:initfam},~$\cF'=\bigcup_{m\in M}\cR_m(\vert\cF\cap\NN^{(m)}\vert)$ 
is an~$M$-complex and, therefore, we can assume without loss of generality 
that $\cF=\cF'$, i.e., that for every~$m\in M$ the set~$\cF\cap\NN^{(m)}$
is an initial segment of~$(\NN^{(m)}, \strictif)$.
Setting~$f_m=\vert\cF\cap\NN^{(m)}\vert$ 
and~$r_m=\vert\cR_M(\vert\cF\vert)\cap\NN^{(m)}\vert$, we of course have
\begin{equation}\label{eq:1613}
	\sum_{m\in M}f_m=\sum_{m\in M}r_m\,.
\end{equation}
Hence, unless~$f_m=r_m$ holds for all~$m$ (in which case we would be done), 
there has to be some maximal~$m^\star\in M$ such 
that~$f_{m^\star}> r_{m^\star}$.

\begin{claim}
	For all~$m\in M$ with~$m\leq m^\star$, we have~$f_m\geq r_m$.
\end{claim}

\begin{proof}
	Since both~$\cF\cap\NN^{(m)}$ and~$\cR_M(\vert\cF\vert)\cap\NN^{(m)}$ 
	are initial segments of~$(\NN^{(m)},\strictif)$, it suffices to find 
	an element~$y\in\cF\cap\NN^{(m)}$ such that for 
	the successor~$y^+$ of~$y$ in~$(\NN^{(m)},\strictif)$, 
	we have~$y^+\not\in\cR_M(\vert\cF\vert)$.
	
	To this end, let~$x=\{x_1,\dots,x_{m^\star}\}\in (\cF\cap\NN^{(m^\star)})\setminus\cR_M(\vert\cF\vert)$, where~$x_1<\dots <x_{m^\star}$, and consider $y=\{x_{m^\star-m+1},\dots,x_{m^\star}\}$.
	Since~$\cF$ is an $M$-complex, we know that~$y\in\cF\cap\NN^{(m)}$.
	Let~$y^+$ be the successor of~$y$ in~$(\NN^{(m)},\strictif)$.
	Then~$\max (y^+\triangle\, x)=\max (y^+\triangle\, y)\in y^+$ and, 
	thus,~$x\strictif y^+$.
	Since~$\cR_M(\vert\cF\vert)$ is an initial segment 
	of $(\bigcup_{\mu\in M}\NN^{(\mu)},\strictif)$ and~$x\notin \cR_M(\vert\cF\vert)$, 
	this entails~$y^+\notin\cR_M(\vert\cF\vert)$, and so~$y$ is as desired.
\end{proof}

The claim and the maximality of~$m^\star$ mean that we have~$f_m\geq r_m$ 
for all~$m\in M$ with~$m\leq m^\star$, while~$f_m\leq r_m$ for all~$m\in M$ 
with~$m>m^\star$.
So when we rewrite~\eqref{eq:1613} in the form
\[
\sum_{m\in M\cap [0, m^\star]} (f_m-r_m)
=
\sum_{m\in M\cap (m^\star, \infty)} (r_m-f_m)\,,
\]
then both sides involve nonnegative summands only. 
Owing to the monotonicity of~$(w_m)_{m\in M}$ we can conclude
\[
\sum_{m\in M\cap [0, m^\star]} w_m(f_m-r_m)
\geq
\sum_{m\in M\cap (m^\star, \infty)} w_m(r_m-f_m)\,,
\]
which implies the desired estimate~$\sum_{m\in M}w_mf_m\geq\sum_{m\in M}w_mr_m$.
\end{proof}

To explain how this lemma is going to be utilised, we recapitulate an 
elegant argument of Frankl~\cite{F:83}, which provides a lower bound 
on $\alpha(m)$ in terms of the function 
\[
\beta(m)
=
\sum_{F\in\cR(m+1)}\frac{1}{\vert F\vert+1}\,.
\]

\begin{lemma}[Frankl]\label{lem:Fr}
For every $m\in\NN_0$, we have~$\alpha(m)\ge\beta(m)$.
\end{lemma}

\begin{proof}
Given a simplicial complex $\cS$ with~$\delta(\cS)\ge m+1$, we need to 
show~$\vert\cS\vert>\beta(m)\vert V\vert$, where~$V=V(\cS)$.
To this end, we consider for every vertex~$x\in V$ the 
sum~$q(x)=\sum_{x\in F\in\cS}\frac{1}{\vert F\vert}$.
Lemma~\ref{lem:Katona} applied to~$M=\NN_0$, the simplicial complex 
\[
\cA=\{F\setminus\{x\}\colon x\in F\in\cS\}\,,
\]
and the weights~$w_j=\frac{1}{j+1}$, 
discloses $q(x)=\sum_{F\in\cA}\frac 1{|A|+1}\ge \beta(|\cA|-1)\ge \beta(m)$ 
for every~$x\in V$. 
By double counting this leads to 
\[
|\cS|
=
1+\sum_{x\in V} q(x)
>
\beta(m) |V|\,. \qedhere
\]
\end{proof}

\begin{exmpl}\label{ex:m=1}
For any nonnegative integer~$d$, we have~$\cR(2^d)=\powerset([d])$, whence
\[
\beta(2^d-1)
=
\sum_{i=0}^d\binom{d}{i}\frac{1}{i+1}
=
\frac{2^{d+1}-1}{d+1}\,.
\]
Using Construction~\ref{constr:base} and the previous lemma, 
we obtain~$\alpha(2^d-1)=\frac{2^{d+1}-1}{d+1}$.
\end{exmpl}

\begin{exmpl}\label{ex:m=2}
For every positive integer~$d$, 
we have~$\cR(2^d-1)=\powerset([d])\setminus\{[d]\}$ and
\[
\beta(2^d-2)
=
\beta(2^d-1)-\frac{1}{d+1}
=
\frac{2^{d+1}-2}{d+1}\,.
\]
In view of Construction~\ref{constr:gen} and Lemma~\ref{lem:Fr}, 
this yields~$\alpha(2^d-2)=\frac{2^{d+1}-2}{d+1}$.
\end{exmpl}

\subsection{A special case}\label{subsec:5}
We now turn to an illustration of our method and establish a slight 
generalisation of the well-known lower bound~$\alpha(5)\ge \frac{13}{4}$, 
which corresponds to the case~$(d,m)=(3,3)$ of Theorem~\ref{thm:main}.
For the sake of comparison, we point out that a direct application of 
Lemma~\ref{lem:Fr} only gives the inferior bound $\alpha(5)\ge \frac{19}6$,
so that some additional averaging argument is needed. 
An analogue of the following statement addressing the case~$d\ge 4$ of 
Theorem~\ref{thm:main} will be formulated 
as Theorem~\ref{thm:meta} below.

\begin{lemma}\label{lem:alpha=5gen}
Let~$\cS$ be a simplicial complex with vertex set~$V(\cS)=V$ 
and let~$f\colon V\lra\NN_0$ be a map.
If for all~$x\in V$, we have~$f(x)+d(x)\geq 6$, then
\[
\frac{1}{2}\sum_{x\in V}f(x)+\vert\cS\vert
>
\frac{13}{4}\vert V\vert\,.
\]
\end{lemma}

\begin{proof}
Assume that the assertion does not hold, and for a fixed vertex set~$V$ 
let~$(\cS,f)$ be a counterexample which minimises~$\vert\cS\vert$.
We commence by observing that~$\cS$ cannot contain sets with more 
than~$3$ elements.
Indeed, if~$\cS$ would contain such sets, let~$F\in\cS$ 
with~$\vert F\vert=m\geq 4$ be maximal.
By deleting~$F$ from~$\cS$, we obtain a simplicial complex~$\cS'$ 
such that~$f(x)+d_{\cS'}(x)\geq 6$ for all~$x\in V\setminus F$ 
and~$d_{\cS'}(x)\geq 2^{m-1}-1>6$ for all $x\in F$.
Hence,~$(\cS',f)$ would be a counterexample 
with~$\vert\cS'\vert<\vert\cS\vert$, a contradiction.

Similarly, one can argue that 
\begin{equation}\label{eq:0038}
	\text{ if a vertex $x$ belongs to a $3$-set of~$\cS$,  	
		then $d(x)\le 6$}. 
\end{equation}
Assume contrariwise that $d(x)\ge 7$ but there is some $xyz\in\cS$.
Let~$\cS'=\cS\setminus\{xyz\}$ and define~$f'\colon V\lra \NN_0$ 
by setting~$f'(y)=f(y)+1$,~$f'(z)=f(z)+1$, and~$f'(v)=f(v)$ for 
all~$v\in V\setminus\{y,z\}$. Clearly, we then still 
have~$f'(v)+d_{\cS'}(v)\geq 6$ for all~$v\in V$ and our minimal choice of $\cS$
entails
\[
\frac{13}{4}|V|
<
\frac{1}{2}\sum_{v\in V}f'(v)+\vert\cS'\vert
=
\frac{1}{2}\sum_{v\in V}f(v)+\vert\cS\vert\,,
\]
which is absurd. 

Having thus proved~\eqref{eq:0038}, we observe that any two $3$-sets in $\cS$
cannot intersect in a single vertex (because this vertex then had degree at
least $7$), and that no vertex can be in three or more $3$-sets of $\cS$.

For~$x\in V$, we 
set~$k(x)=\frac{1}{2}f(x)+\sum_{x\in F\in\cS}\frac{1}{\vert F\vert}$.
Note that this is the `contribution of $x$' to the term we want to find 
a lower bound for, which can be rewritten as $1+\sum_{x\in V}k(x)$.

Define a {\it quartette} to be a~quadruple $Q\in V^{(4)}$ such 
that~$\vert Q^{(3)}\cap\cS\vert\geq 2$.
We contend that every vertex~$x$ which is not contained in any quartette 
satisfies~$k(x)>\frac{13}{4}$.
Our discussion of triples in $\cS$ 
shows that $x$ can belong to at most one triple.
Let us view $k(x)$ 
as a sum of $f(x)$ terms equal to $\frac12$ and $d(x)$ terms of the 
form $\frac 1{|F|}$, so that there are altogether at least $6$ summands.
Among them there are $1$, at least four equal to $\frac12$, and a sixth one 
that is at least $\frac13$.
For this reason, we have indeed 
$k(x)\ge 1+\frac42+\frac13=\frac{10}3>\frac{13}4$.

Next we consider any quartette~$Q$. Again by our discussion of triples in $\cS$,
there are two vertices in $Q$, say $u$ and $v$, which belong to exactly two 
triples, while the two other ones, say $x$ and $y$, are in exactly one triple. 
Consequently, $Q$ is disjoint to all other quartettes. Moreover, the argument 
from the previous paragraph discloses~$k(x), k(y)\ge \frac{10}3$.
Similarly, we obtain $k(u), k(v)\ge 1+\frac 32+\frac 23=\frac{19}6$, and a simple 
addition leads to $\sum_{z\in Q} k(z)\ge 13$.

Let~$\cQ$ be the set of all quartettes, and let~$H$ be the set of vertices 
which are not contained in any quartette. Recalling that the quartettes are 
disjoint, we obtain
\[
\sum_{v\in V}k(v)
=
\sum_{Q\in\cQ}\sum_{v\in Q}k(v)+\sum_{v\in H}k(v)
\geq 
13\vert\cQ\vert+\frac{13}{4}\vert H\vert=\frac{13}{4}\vert V\vert\,,
\]
whence 
\[
\frac{1}{2}\sum_{v\in V}f(v)+\vert\cS\vert
=
1+\sum_{v\in V}k(v)
>
\frac{13}{4}\vert V\vert\,.
\]
But this contradicts the fact that~$(\cS,f)$ is a counterexample.
\end{proof}

\begin{cor}\label{cor:alpha=5}
We have~$\alpha(5)=\frac{13}{4}$.
\end{cor}

\begin{proof}
The upper bound $\alpha(5)\le\frac{13}4$ can be seen by plugging $d=m=3$
into Construction~\ref{constr:gen}. Now suppose conversely that $\cS$ 
is a simplicial complex with~$\delta(\cS)\geq6$.
Lemma~\ref{lem:alpha=5gen}, applied to~$\cS$ and the constant 
function~$f\colon V(\cS)\lra\{0\}$, 
yields the desired bound~$\vert\cS\vert>\frac{13}{4}\vert V\vert$.
\end{proof}

\section{Local lemmata}\label{sec:local}

In order to adapt Frankl's Lemma~\ref{lem:Fr} to the situation considered in 
Theorem~\ref{thm:main}, we need lower bounds 
on $\sum_{F\in\cA}\frac{1}{\vert F\vert+1}$ for simplicial complexes $\cA$
of size $|\cA|\ge 2^d-m+1$, where~$d\ge m\ge 3$ are given. We commence with the case 
that the ground set $V(\cA)$ has size $d$.    

\begin{lemma}\label{lem:local}
Let~$d$, $m$, $s$ be nonnegative integers such that $d\ge \max\{s, 2\}$ 
and~$m\ge s+2$. Further, let~$\cA$ be a simplicial complex 
with~$\vert V(\cA)\vert=d$ 
and~$\vert\cA\vert\geq 2^d-m+1$. If $\cA$ contains exactly~$d-s$ sets 
of size~$d-1$, then
\begin{align}\label{eq:locweight1}
	\sum_{F\in\cA}\frac{1}{\vert F\vert+1}
	\geq
	\frac{2^{d+1}-2}{d+1}-\Big(\frac{s}{d}+\frac{m-2-s}{d+1-r}\Big)
\end{align}
holds 	for some nonnegative integer~$r$ such that~$2^r\leq m-1$ and~$r\leq s$.
\end{lemma}

\begin{proof}
We can assume that the vertex set $V=V(\cA)$ is not in $\cA$, because 
removing~$V$ from~$\cA=\powerset(V)$ would keep the assumptions valid 
and decrease the left side of our estimate. 

Now we consider the simplicial 
complex~$\cB=\{F\subseteq V\colon \vert F\vert\leq d-2 \text{ or } F\in\cA\}$,
which arises from $\powerset(V)$ by removing $V$ itself and~$s$ sets of size~$d-1$.
As in Example~\ref{ex:m=2} we see that
\begin{align*}
	\sum_{F\in\cB}\frac{1}{\vert F\vert +1}
	=
	\frac{2^{d+1}-2}{d+1}-\frac{s}{d}\,.
\end{align*}
So due to $\cA\subseteq\cB$, it remains to show 
that~$\sum_{F\in\cB\setminus\cA}\frac{1}{\vert F\vert+1}\leq\frac{m-2-s}{d+1-r}$ 
for some nonnegative integer~$r$ with~$2^r\leq m-1$ and~$r\leq s$.

In the special case~$\cB=\cA$, this is clear, so we can take a 
set~$A_\star\in\cB\setminus\cA$ such that~$\vert A_\star\vert$ is minimal.
Write~$\vert A_\star\vert=d-r$ with~$r\geq 2$ and notice 
that~$\sum_{F\in\cB\setminus\cA}\frac{1}{\vert F\vert+1}\leq\frac{\vert\cB\setminus\cA\vert}{d-r+1}$.
Further, one readily computes
\begin{align}\label{eq:notinS}
	\vert\cB\setminus\cA\vert
	&=
	|\cB|-\vert\cA\vert
	\leq 
	(2^d-s-1)-(2^d-m+1)
	= 
	m-s-2\,.
\end{align}

Thus, we can finish the proof by showing~$r\leq s$ and~$2^r\leq m-1$.
To this aim, we analyse the sets in~$\powerset(V)$ which must be missing 
from~$\cA$ because they are supersets of~$A_\star$.
First, each of the~$r$ elements $x\in V\setminus A_\star$ gives 
a~$(d-1)$-set~$V\setminus\{x\}$ containing~$A_\star$; together with the 
hypothesis $|V^{(d-1)}\sm\cA|=s$ this proves $r\le s$.
				The second part is shown by 
\[
2^r
=
\big|\{F\colon A_\star\subseteq F \subseteq V\}\big|
\le 
|\powerset(V)\sm\cA|
\le
m-1\,. \qedhere
\]
\end{proof}

The following lemma will later be used to argue that if a simplicial complex uses a ground set larger 
than necessary, then its weight is sufficiently large.

\begin{lemma}\label{lem:unified}
Given integers $d\ge m\ge 3$, let $\cA$ be a simplicial complex 
with~$|V(\cA)|\geq d+1$ and~$|\cA|\geq 2^d-m+1$.
\begin{enumerate}[label=\alabel]
	\item\label{it:lua} We have $\sum_{F\in\cA}\frac{1}{|F|+1}\ge 
	\frac{2^{d+1}-m}{d+1}$.
	\item\label{it:lub} If $\cA$ contains exactly~$d-s-1$ sets of size~$d-1$ 
	for some~$s\in [0, m-2]$, then there is an integer~$r$ 
	with~$0\leq r\leq s+1$,~$2^r\leq m$, and
	\[
	\sum_{F\in\cA}\frac{1}{|F|+1}
	\geq
	\frac{2^{d+1}-2}{d+1}
	-
	\Big(\frac{s+1}{d}+\frac{m-2-s}{d+1-r}-\frac{1}{2}\Big)\,.
	\]
\end{enumerate}
\end{lemma}

\begin{proof}
Arguing for fixed $d$, $m$ by induction on $|\cA|+|V(\cA)|$, we suppose that both 
statements hold for all simplicial complexes $\cA'$ satisfying $|V(\cA')|\ge d+1$, 
$|\cA'|\geq 2^d-m+1$,
and~$|\cA'|+|V(\cA')|<|\cA|+|V(\cA)|$. If some set $F\in \cA$ has size at 
least $d$, then $|\cA|\ge 2^d$, so that we can simply delete a maximal such 
set and apply induction. Therefore, we can assume~${|F|<d}$ for all~$F\in \cA$.
Remark~\ref{rem:initfam} allows us to suppose further, that
\begin{equation}\label{eq:1841}
	\cA\cap\NN^{(\mu)}=\cR_\mu(|\cA\cap\NN^{(\mu)}|)
\end{equation}
for all~$\mu\in\NN$. 

Define the integer $s\le d-1$ by $|\cA\cap\NN^{(d-1)}|=d-s-1$. Let us first look
at the special case $s<0$, where~\ref{it:lub} is not applicable and we only have 
to deal with~\ref{it:lua}. The case~$\mu=d-1$ of~\eqref{eq:1841} 
yields $([d]\sm\{i\})\in\cA$ for all~$i\in [d]$, 
whence~$\powerset([d])\sm\{[d]\}\subseteq \cA$. As we saw in Example~\ref{ex:m=2}, 
this entails~$\sum_{F\in\cA}\frac{1}{|F|+1}\ge \frac{2^{d+1}-2}{d+1}$, which is 
stronger than necessary. 

It remains to study the case $s\ge 0$, where~\eqref{eq:1841} reveals  
\[
\cA\cap\NN^{(d-1)}=\{[d]\setminus \{i\}\colon s+2\leq i\leq d\}\,.
\]
Moreover, when applied to~$\mu=1$~\eqref{eq:1841} tells 
us $V(\cA)=[D]$ for some $D>d$. Now we set
\[
M=\{2,\dots,d-2\}
\quad\text{ as well as }\quad
\gM=\bigcup_{\mu\in M}\NN^{(\mu)}\,,
\]
and try to `simplify' $\cA\cap\gM$. 
To this end, we set $\cT=\cR_M(|\cA\cap\gM|)$ and consider the set system
\begin{align*}
	\cB=\{F\in\cA\colon |F|\in\{0, 1, d-1\}\}\cup\cT\,.
\end{align*}

We contend that~$\cB$ is a simplicial complex. Since $\powerset([D])\cap\gM$
is an initial segment of~$(\gM, \strictif)$ which contains $\cA\cap \gM$, it 
contains $\cT$ as well, and therefore we have $\cB\subseteq\powerset([D])$.
Due to the choice of~$\cT$, it is only left to argue that for all~$F'\subseteq F\in\cA\cap\NN^{(d-1)}$ with~$\vert F'\vert\geq 2$, it follows that~$F'\in\cT$.
For this, observe that also~$\cQ=\{F\in \powerset([d])\cap\gM\colon [s+2,d]\not\subseteq F\}$
is an initial segment of~$(\gM, \strictif)$. 
Thus, $\cQ\subseteq \cA$ implies $\cQ\subseteq \cT$, which completes the proof that~$\cB$ is a simplicial 
complex.
As Lemma~\ref{lem:Katona} 
tells us~$\sum_{F\in \cA}\frac 1{|F|+1}\ge \sum_{F\in \cB}\frac 1{|F|+1}$,
we can henceforth assume that~$\cA=\cB$.
The main gain of this transformation is 
that now the set $\cT=\cA\cap\gM$ satisfies the trichotomy 
\[
\cT\subsetneq \powerset([d])\cap\gM
\quad \text{ or } \quad
\cT= \powerset([d])\cap\gM
\quad \text{ or } \quad
\cT\supsetneq \powerset([d])\cap\gM\,,
\]
which due to $|\powerset([d])\cap\gM|=2^d-2d-2$ suggests the following 
case distinction. 

\smallskip

{\hskip2em \it First case: $|\cT|\ge 2^d-2d-2$ and $|\cT|+D\ge 2^d-d$.}

\smallskip

Assume first that $|\cA| > 2^d-m+1$. If $|\cT| > 2^d-2d-2$, then 
$\cT\sm (\powerset([d])\cap\gM)$ is non-empty and we can simply 
delete an inclusion-maximal element of this set and apply induction. 
On the other hand, if $|\cT|=2^d-2d-2$, or in other 
words $\cT=\powerset([d])\cap\gM$, then $D\ge (2^d-d)-|\cT|=d+2$, so that 
we can delete $\{D\}$ and apply induction. 

Thus it remains to consider the case $|\cA|=2^d-m+1$. As
\begin{align*}
	d-s-1
	&=
	|\cA\cap\NN^{(d-1)}|
	=
	|\cA|-(|\cT|+D+1) \\
	&\le
	(2^d-m+1)-(2^d-d+1)
	=d-m
\end{align*}
yields $s\ge m-1$, part~\ref{it:lub} is mute and we only need to 
establish part~\ref{it:lua}. 

To this end, we observe that $\cE=\{F\subseteq [d]\colon |F|\le d-2\}$
is a subset of $\cA$ and the calculation in Example~\ref{ex:m=2} shows
\[
\sum_{F\in \cE}\frac 1{|F|+1}
=
\frac{2^{d+1}-2}{d+1}-1\,.
\]
As the set system $\cA\sm\cE$ contains $\{d+1\}$ and $d-m+1$ further sets, 
each of which has size at most $d$, we infer 
\[
\sum_{F\in \cA}\frac 1{|F|+1}
\ge
\frac{2^{d+1}-2}{d+1}-1+\frac12+\frac{d-m+1}{d+1}
=
\frac{2^{d+1}-m}{d+1}+\frac12-\frac 2{d+1}\,.
\]
Because of $d\ge 3$, this implies the desired estimate.

\smallskip

{\hskip2em \it Second case: $|\cT|\le 2^d-2d-3$ or $|\cT|+D\le 2^d-d-1$.}

\smallskip

Suppose first that $D>d+1$. Now $|\cT|\le 2^d-2d-3$ is surely true, which in turn implies~$\cT\subsetneq \powerset([d])\cap\gM$. 
Thus, we can remove 
the set $\{D\}$ from $\cA$ and add an inclusion-minimal set from~$(\powerset([d])\cap\gM)\sm\cT$ instead, to get a new simplicial complex~$\cC$ with~$V(\cC)=[D-1]$.
Since $|\cC|=|\cA|$ and $|V(\cC)|<|V(\cA)|$,
the induction hypothesis applies to $\cC$, and due to $\sum_{F\in\cA}\frac 1{|F|+1}\ge \sum_{F\in\cC}\frac 1{|F|+1}$, we are done. 

It remains to study the case $D=d+1$.
Now~$\vert\cT\vert\leq 2^d-2d-2$, whence~$\cT\subseteq\powerset([d])\cap\gM$ and $|\cT|+D\le 2^d-d-1$. Due to 
\begin{align*}
	d-s-1
	&=
	|\cA\cap\NN^{(d-1)}|
	=
	|\cA|-(|\cT|+D+1) \\
	&\ge
	(2^d-m+1)-(2^d-d)
	=d-m+1
\end{align*}
we have $s\le m-2$. Thus part~\ref{it:lub} follows from Lemma~\ref{lem:local}
applied to the numbers~$m+1$,~$s+1$, and the simplicial 
complex $\cA\sm\{\{d+1\}\}$ here in place of $m$, $s$, and $\cA$ there.
Moreover, Corollary~\ref{fa:ineq2} yields part~\ref{it:lua}.
\end{proof}

On some occasions the following very weak estimate will suffice.

\begin{lemma}\label{lem:weightboundsimple}
Let~$d\geq 4$ be an integer and let~$\cA$ be a simplicial complex. 
If~$\vert\cA\vert\geq 2^d-2d+2$, then
\[
\sum_{F\in\cA}\frac{1}{\vert F\vert+1}
\geq
\frac{2^{d+1}-1}{d+1}-\frac{2d-2}{3}\,.
\]
\end{lemma}

\begin{proof}
The case~$M=\NN_0$ of Lemma~\ref{lem:Katona} allows us to assume that 
$\cA=\cR(2^d-2d+2)$. As $d\ge 4$ yields $2^d-2d+2\in (2^{d-1}, 2^d)$,
we then have $V(\cA)=[d]$. It remains to 
recall $\sum_{F\subseteq [d]}\frac{1}{\vert F\vert+1}=\frac{2^{d+1}-1}{d+1}$ 
and to observe that each of the at most $2d-2$ sets 
in $\powerset([d])\sm\cA$ has size at least $2$.
\end{proof}

\section{Mountains}\label{sec:mountains}

We call $\NN_{\geq 2}$-complexes (cf.\ Definition~\ref{def:Mcomplex}) {\it mountains}. 
The following result on mountains generalises Theorem~\ref{thm:main}.

\begin{theorem}\label{thm:meta}
Given integers~$d$ and~$m$ with~$d\geq4$ and~$d\geq m\geq3$, let
\begin{itemize}
\item $V$ be a set of vertices,
\item $\fB_1,\dots,\fB_t\subseteq\powerset(V)$ be mountains,
\item and let~$f\colon V\lra\NN_0$ be a map.
\end{itemize}
If for all~$x\in V$, we have
\begin{align*}
f(x)+\sum_{\tau=1}^td_{\fB_{\tau}}(x)\geq 2^d-m\,,
\end{align*}
and for all~$F\in\bigcup_{\tau\in [2, t]}\fB_{\tau}$, we have~$\vert F\vert<d$, 
then
\begin{align*}
\sum_{\tau=1}^t\vert\fB_{\tau}\vert+\sum_{x\in V}\Big(\frac{f(x)}{2}+1\Big)
\geq
\frac{2^{d+1}-m}{d+1}\vert V\vert\,.
\end{align*}
\end{theorem}

Let us first derive Theorem~\ref{thm:main} from this result.

\begin{proof}[Proof of Theorem~\ref{thm:main} assuming Theorem~\ref{thm:meta}]
Let integers~$d\geq m\geq 1$ be given.
Recall that Construction~\ref{constr:gen} provides the upper bound, so it 
remains to show~$\alpha(2^d-m)\geq\frac{2^{d+1}-m}{d+1}$.
Further note that due to Examples~\ref{ex:m=1} and~\ref{ex:m=2}, 
we may assume~$d\geq m\geq 3$, and because of Corollary~\ref{cor:alpha=5}, 
we can even restrict ourselves to~$d\geq4$.

Now let~$\cS$ be a simplicial complex with~$\delta(\cS)\geq 2^d-m+1$.
We are to prove the lower 
bound~$\vert\cS\vert>\frac{2^{d+1}-m}{d+1}\vert V(\cS)\vert$.
Theorem~\ref{thm:meta} applied to
\begin{itemize}
\item $V=V(\cS)$,
\item $t=1$ and the mountain~$\fB_1=\{F\in\cS:\vert F\vert\geq 2\}$,
\item and the zero function~$f\colon V\lra\{0\}$
\end{itemize}
yields 
\[
\vert\fB_1\vert+\vert V\vert
\geq
\frac{2^{d+1}-m}{d+1}\vert V\vert\,.
\]
Since~$\vert\cS\vert=\vert\fB_1\vert+\vert V\vert+1$, we are done.
\end{proof}

The rest of this section is dedicated to the proof of Theorem~\ref{thm:meta}.
Fix~$d$,~$m$, and~$V$ as there and assume that the theorem does not 
hold. Choose a counterexample consisting of 
mountains~$\fB_1,\dots,\fB_t\subseteq\powerset(V)$ and a 
map~$f\colon V\lra\NN_0$ such that~$|\fB_1|$ is minimal.
The argument which gave~\eqref{eq:0038} generalises as follows. 

\begin{claim}\label{cl:smalldegvtsexist}
If~$M\in\fB_1$ is inclusion-maximal, then there are at least three 
vertices~$x\in M$ with~$d_{\fB_1}(x)\leq 2^d-m$.
In particular,~$\vert F\vert\leq d$ for all~$F\in\fB_1$.
\end{claim}

\begin{proof}
Let~$M'=\{x\in M\colon d_{\fB_1}(x)\leq 2^d-m\}$ and 
define~$\fB_1'=\fB_1\setminus \{M\}$ and
\begin{align*}
f'(x)=\begin{cases}
	f(x)+1 &\text{ if } x\in M'\\
	f(x) &\text{ if } x\in V\setminus M'\,.
\end{cases}
\end{align*}
Now~$\fB_1',\fB_2,\dots,\fB_t$ satisfy the conditions in the statement of 
the theorem. Since we chose the counterexample such that~$\vert\fB_1\vert$ is 
minimal and~$\vert\fB_1'\vert<\vert\fB_1\vert$, we infer that
\begin{align*}
\Big(\frac{\vert M'\vert}{2}-1\Big)
+\sum_{\tau=1}^t\vert\fB_{\tau}\vert+\sum_{x\in V}\Big(\frac{f(x)}{2}+1\Big)
\geq
\frac{2^{d+1}-m}{d+1}\vert V\vert\,.
\end{align*}
Thus, if~$\vert M'\vert\leq 2$, then the desired inequality would hold 
for~$\fB_1,\dots,\fB_t$ and~$f$, contradicting that these form a counterexample.
\end{proof}

For every~$x\in V$, set
\begin{align*}
\mu(x)=f(x)+\vert\{(y,\tau)\in V\times[t]\colon xy\in\fB_{\tau}\}\vert
\end{align*}
and
\begin{align*}
k(x)
=
1+\frac{f(x)}{2}+\sum_{\tau=1}^t\sum_{x\in F\in\fB_{\tau}}\frac{1}{\vert F\vert}\,.
\end{align*}

The left side of the inequality we are dealing with rewrites as $\sum_{x\in V}k(x)$.
Following the basic strategy of~\S\ref{subsec:5}, we analyse how the individual summands $k(x)$ compare to the alleged average of $\frac{2^{d+1}-m}{d+1}$. 

\begin{claim}\label{cl:bigandsmallmu}
For~$x\in V$, the following statements hold.
\begin{enumerate}
\item\label{it:munottoosmall} We have~$\mu(x)\geq d$.
\item\label{it:smallmu} If~$\mu(x)=d$, then
\begin{enumerate}
	\item\label{it:smallmuf} $f(x)=d_{\fB_2}(x)=\dots=d_{\fB_t}(x)=0$,
	$d_{\fB_1}(x)\ge 2^d-m$,
		\item\label{it:smallmuk} and there are integers~$r$ and~$s$ 
	with~$0\leq r\leq s\leq m-2$ and~$2^r\leq m-1$ such that
	\[
	k(x)
	\geq
	\frac{2^{d+1}-2}{d+1}
	-
	\Big(\frac{s}{d}+\frac{m-2-s}{d+1-r}\Big)
			\ge
	\frac{2^{d+1}-m}{d+1}-\frac{(d-2)(d-3)}{4d}
	\] 
	and the number of~$d$-sets in~$\fB_1$ containing~$x$ is exactly~$d-s$.
\end{enumerate}
\item\label{it:bigmu} If~$\mu(x)\geq d+1$, then
\begin{enumerate}
	\item\label{it:bigmuk} $k(x)\geq\frac{2^{d+1}-m}{d+1}$,
	\item\label{it:bigmukspecial} and if there is some~$s\in[0,m-2]$ 
	such that the number of~$d$-sets in~$\fB_1$ containing~$x$ is 
	exactly~$d-s-1$, then there is an integer~$r$ with~$0\leq r\leq s+1$ 
	and~$2^r\leq m$ satisfying 
	\[
	k(x)
	\geq
	\frac{2^{d+1}-2}{d+1}
	-
	\Big(\frac{s+1}{d}+\frac{m-2-s}{d+1-r}-\frac{1}{2}\Big)\,.
	\]
\end{enumerate}
\end{enumerate}
\end{claim}

\begin{proof}
We commence by defining vertex-disjoint simplicial 
complexes~$\cA_1,\dots,\cA_{t+f(x)}$ such that
\begin{itemize}
\item $\cA_{\tau}\cong\{\emptyset\}\cup\{F\setminus\{x\}
\colon x\in F\in\fB_{\tau}\}$ for~$\tau\in[t]$,
\item $\cA_{t+i}\cong\powerset([1])$ for~$i\in[f(x)]$.
\end{itemize}
Then~$\cA=\cA_1\cup\dots\cup\cA_{t+f(x)}$ is a simplicial complex 
with~$\vert V(\cA)\vert=\mu(x)$,
\[
\vert\cA\vert
=
1+f(x)+\sum_{\tau\in[t]}d_{\fB_{\tau}}(x)
\geq 
2^d-m+1\,,
\]
and~$\sum_{F\in\cA}\frac{1}{\vert F\vert+1}=k(x)$.
Due to $|F|<d$ for all $F\in\bigcup_{\tau\in[2, t]}\fB_\tau$ 
and Claim~\ref{cl:smalldegvtsexist}, we have 
\[|\cA\cap V^{(d-1)}|=|\{F\in \fB_1\cap V^{(d)}\colon x\in F\}|
\quad \text{ and } \quad
\cA\cap V^{(d)}=\vn\,.
\]Therefore, Lemma~\ref{lem:unified} implies clause~\eqref{it:bigmu}	and we can 
assume~$\mu(x)\leq d$ throughout the remainder of the argument. 

Set~$n_i=\vert V(\cA_i)\vert$ for~$i\in[t+f(x)]$.
Owing to~$\sum_{i\in[t+f(x)]}n_i=\mu(x)\leq d$ and 
\[
\sum_{i\in[t+f(x)]}(2^{n_i}-1)
\ge
\sum_{i\in[t+f(x)]}(\vert\cA_i\vert-1)
=
\vert\cA\vert-1\geq 2^d-m\geq2^d-d\,,
\]
Fact~\ref{fa:ineq1} guarantees that there is an index~$i\in[t+f(x)]$ such 
that~$n_i=d$ and~$n_j=0$ for all~$j\in[t+f(x)]\setminus\{i\}$.
This immediately implies~$f(x)=0$ and~$\mu(x)=d$; 
so~\eqref{it:munottoosmall} is proved and for~\eqref{it:smallmuf}
it suffices to show~$i=1$, which is what we shall do next.

Define~$s$ such that
\[
d-s
=
\vert\{F\in\fB_i\colon x\in F\text{ and }\vert F\vert=d\}\vert\,.
\]
In other words,~$d-s$ is the number of~$(d-1)$-sets in~$\cA$, 
whence~$d-s\leq\binom{\vert V(\cA)\vert}{d-1}=\binom{d}{d-1}=d$ and~$s\geq 0$.
Moreover, $\cA\cap V^{(d)}=\vn$ yields
\[
2^d-m+1
\leq
\vert\cA\vert
\leq 
\sum_{j=0}^{d-2}\binom{d}{j}+(d-s)=2^d-s-1\,,
\]
whence~$m\geq s+2$ and~$d-s\geq d-m+2\geq2$. In particular, $\fB_i$ contains a $d$-set, 
which proves~$i=1$.			

Furthermore, we have now shown that~$\cA$ satisfies all the conditions to apply 
Lemma~\ref{lem:local}, which in turn implies~\eqref{it:smallmuk} apart from the 
second inequality. To confirm this final part, we observe 
that~$2^r\leq m-1\leq d-1<2^{d-2}$ gives~$r\leq d-3$, wherefore
\begin{align*}
\frac{s}{d}+\frac{m-2-s}{d+1-r}
&\le
\frac{s}{4}+\frac{m-2-s}{4}
=
\frac{m-2}{4}
=
\frac{m-2}{d+1}+\frac{(m-2)(d-3)}{4(d+1)} \\
&\le 
\frac{m-2}{d+1}+\frac{(d-2)(d-3)}{4d}\,. \qedhere
\end{align*}
\end{proof}

In particular, all `problematic vertices' $x$ with $k(x)<\frac{2^{d+1}-m}{d+1}$
satisfy $\mu(x)=d$. This allows us to study their `neighbourhoods', which are 
completely contained in~$\fB_1$. We are thereby led to the notion of a 
conglomerate, which can be thought of as an adaptation of quartettes 
from~\S\ref{subsec:5} to the present circumstances.  

\begin{dfn}\label{dfn:cong}
A {\it conglomerate} is a set~$K\in V^{(d+1)}$ for which there is a 
vertex~$x\in K$ with 
\[
\vert\{A\colon x\in A\subseteq K\}\setminus\fB_1\vert\leq m\,.
\]
\end{dfn}

Roughly speaking, `most' subsets of a conglomerate are in $\fB_1$. This can be made precise as follows. 

\begin{claim}\label{cl:interactionconglob1}
If~$K$ denotes a conglomerate, then
\begin{enumerate}
\item\label{it:upperlayerconglo} $\vert K^{(d)}\cap\fB_1\vert\geq d-m+2\geq 2$,
\item\label{it:b1closetoconglo} $\vert\powerset(K)\setminus\fB_1\vert\leq d+2m$,
\item\label{it:b1closetocongloeverywhere} for every~$z\in K$, 
we have $\vert\{A\colon z\in A\subseteq K\}\setminus\fB_1\vert \leq2m-1$,
\item\label{it:1415} and $K^{(2)}\subseteq \fB_1$.
\end{enumerate}
\end{claim}

\begin{proof}
Pick~$x\in K$ such that setting~$\cA=\{A\colon x\in A\subseteq K\}$ 
we have~$\vert\cA\setminus\fB_1\vert\leq m$ (which is possible by 
Definition~\ref{dfn:cong}).
Using that~$\fB_1$ does not contain any set of size~$d+1$, we obtain
\begin{align*}
2^d-m
&\leq
\vert\cA\cap\fB_1\vert
=
\sum_{i=2}^{d}\vert\cA\cap\fB_1\cap K^{(i)}\vert\\
&\leq
\sum_{i=2}^{d-1}\binom{d}{i-1}+\vert\fB_1\cap K^{(d)}\vert\\
&=
2^d-d-2+\vert\fB_1\cap K^{(d)}\vert\,,
\end{align*}
which proves~\eqref{it:upperlayerconglo}.

Now we consider the decomposition
\begin{align*}
\powerset(K)\setminus\fB_1
=
\{A\subseteq K\colon \vert A\vert\leq 1\}
&\dcup
\{A\subseteq K\colon \vert A\vert\geq 2\text{ and }x\in A\not\in\fB_1\}\\
&\dcup
\{A\subseteq K\colon\vert A\vert\geq 2\text{ and }x\notin A\not\in\fB_1\}\,.
\end{align*}
The first set on the right side contains~$d+2$ elements, and the second one at 
most~$m-1$ (by the choice of~$x$). Moreover, since~$\fB_1$ is a mountain, 
the map $A\longmapsto A\dcup\{x\}$ is an injection from the third set into 
the second. Therefore, also the third set contains at most~$m-1$ elements.
Combining these bounds we 
infer~$\vert\powerset(K)\setminus\fB_1\vert\leq(d+2)+2(m-1)=d+2m$, 
which proves~\eqref{it:b1closetoconglo}.

Given any~$z\in K$, among the at most~$d+2m$ sets in~$\powerset(K)\setminus\fB_1$
there are at least~$d+1$ not containing~$z$, namely~$\vn$ and the~$d$ sets~$\{y\}$		with~$y\in K\sm\{z\}$.
Thus, we indeed have $\vert\{A\colon z\in A\subseteq K\}\setminus\fB_1\vert\le (d+2m)-(d+1)=2m-1$,
which proves~\eqref{it:b1closetocongloeverywhere}.

Finally, assume that contrary to~\eqref{it:1415} there is some 
pair $yz\in K^{(2)}\sm\fB_1$. Since none of the~$2^{d-1}$ sets~$A$ 
with $yz\subseteq A\subseteq K$ can belong to~$\fB_1$, we then obtain 
\[
|\{A\colon z\in A\subseteq K\}\sm\fB_1|\ge 2^{d-1}>2d-1\ge 2m-1\,,
\] 
which contradicts~\eqref{it:b1closetocongloeverywhere}. 
\end{proof}

\begin{claim}\label{cl:weightheavyvts}
If a vertex~$x\in V$ satisfies $\mu(x)=d$, then there is a unique conglomerate 
containing $x$, namely the set $K=\{x\}\cup\{y\in V\colon xy\in\fB_1\}$.
\end{claim}

\begin{proof}	
Claim~\ref{cl:bigandsmallmu}\eqref{it:smallmuf} discloses~$|K|=d+1$. 
As every set $A\in\fB_1$ with $x\in A$ needs to be a subset of $K$, 
we also 
have $|\{A\colon x\in A\subseteq K\}\setminus\fB_1|=2^d-d_{\fB_1}(x)\leq m$,
which proves that~$K$ is indeed a conglomerate. 
The uniqueness of $K$ is immediate from 
Claim~\ref{cl:interactionconglob1}\eqref{it:1415}.
\end{proof}

\begin{claim}\label{cl:largemtnedgeconglo}
Let~$K$ be a conglomerate and~$D\in\fB_1\cap V^{(d)}$. If~$D\not\subseteq K$, 
then~$|D\sm K|\ge 3$ and every $x\in D\cap K$ satisfies 
$k(x)\ge \frac{2^{d+1}-1}{d+1}+\frac 12$. 
\end{claim}

\begin{proof}
Assume for the sake of contradiction that~$\vert D\sm K\vert\in\{1, 2\}$.
Since~$\fB_1$ contains no sets of size~$d+1$, we know that~$D$ is inclusion-maximal
in~$\fB_1$ and Claim~\ref{cl:smalldegvtsexist} guarantees that some~$x\in D\cap K$ 
satisfies~$d_{\fB_1}(x)\leq 2^d-m$. On the other hand, 
Claim~\ref{cl:interactionconglob1}\eqref{it:b1closetocongloeverywhere} informs us 
that there are at least~$2^d-(2m-1)$ sets in~$\powerset(K)\cap\fB_1$ 
containing~$x$. Therefore, the set
\[
Q=\{A\subseteq D\colon x\in A\text{ and }\vert A\vert\geq 2
\text{ and }A\not\subseteq K\}
\] 
satisfies~$\vert Q\vert\leq m-1$. But if~$\vert D\sm K\vert=1$, 
then~$\vert Q\vert=2^{d-2}\geq d\geq m$, while~$\vert D\sm K\vert=2$
would imply~$\vert Q\vert= 3\cdot 2^{d-3}\geq 2^{d-2}\geq m$. 
This contradiction concludes the proof of~$|D\sm K|\ge 3$. 

Now let $x\in D\cap K$ be arbitrary.
Using again that Claim~\ref{cl:interactionconglob1}\eqref{it:b1closetocongloeverywhere} states 
\[
\vert\{A\colon x\in A\subseteq K\}\cap\fB_1\vert\geq 2^d-2m+1\,,
\]
we learn that the simplicial complex 
\[
\cA=\{\vn\}\cup\{A\sm\{x\}\colon x\in A\in(\fB_1\cap\powerset(K))\}
\]
satisfies~$\vert\cA\vert\geq 2^d-2d+2$, 
whence by Lemma~\ref{lem:weightboundsimple}
\[
\sum_{A\in\cA}\frac{1}{\vert A\vert+1}
\geq
\frac{2^{d+1}-1}{d+1}-\frac{2d-2}{3}\,.
\]
Due to $|D\sm K|\ge 3$, the set system 
\[
\cA'=\{A\setminus\{x\}\colon x\in A\subseteq D\text{ and }A\not\subseteq K\}
\] 
contains at least three~$1$-sets. Moreover, the number of $2$-sets in $\cA'$ is 
\[
\binom{|D|-1}2-\binom{|D\cap K|-1}2
\ge 
\binom {d-1}2-\binom{d-4}2=3d-9\ge 2d-5\,.
\]
For these reasons, we have indeed
\begin{align*}
k(x)
&\geq
\sum_{A\in\cA}\frac{1}{\vert A\vert+1}+\sum_{A\in\cA'}\frac{1}{\vert A\vert+1}\\
&\geq
\frac{2^{d+1}-1}{d+1}-\frac{2d-2}{3}+\frac{3}{2}+\frac{2d-5}{3}\\
&=
\frac{2^{d+1}-1}{d+1}+\frac{1}{2}\,. \qedhere
\end{align*}
\end{proof}

\begin{claim}\label{clm:1957}
If a conglomerate $K$ is disjoint to all other conglomerates, then
\[
\sum_{x\in K}k(x)\ge 2^{d+1}-m\,.
\]
\end{claim}

\begin{proof}
Set 
\[
Q=\{x\in K\colon K\setminus\{x\}\not\in\fB_1\}
\]
and suppose first that some vertex~$x\in K\sm Q$ satisfies~$\mu(x)=d$. 
Due to Claim~\ref{cl:weightheavyvts}, we know $K=\{x\}\cup\{y\in V\colon xy\in\fB_1\}$ 
and, therefore, Claim~\ref{cl:bigandsmallmu}\eqref{it:smallmuf} tells us that at least~$2^d-m$ sets in~$\powerset(K)\cap\fB_1$ 
contain~$x$. Further,~$K\setminus\{x\}\in\fB_1$ and~$\fB_1$ being a mountain 
means that~$\powerset(K)\cap\fB_1$ contains~$2^d-d-1$ sets which do not 
contain~$x$.
We infer that
\[
\vert\powerset(K)\cap\fB_1\vert\geq 2^{d+1}-d-m-1\,.
\]
Since~$k(z)\geq1+\sum_{z\in F\in\powerset(K)\cap\fB_1}\frac{1}{\vert F\vert}$ 
for every~$z\in K$, we have 
\[
\sum_{z\in K}k(z)
\geq
\vert K\vert+\vert\powerset(K)\cap\fB_1\vert\geq 2^{d+1}-m\,,
\]
as desired. So from now on we may focus on the case that $\mu(x)\ge d+1$ 
for all $x\in K\sm Q$. 

Our claim is clear if $k(z)\ge \frac{2^{d+1}-m}{d+1}$ holds for every~$z\in K$, whence we may assume that $k(z)<\frac{2^{d+1}-m}{d+1}$ holds for some~$z\in K$. 
Claim~\ref{cl:bigandsmallmu}\eqref{it:bigmuk} yields $\mu(z)=d$
and thus we have~$z\in Q$. Summarising this paragraph, 
\begin{equation}\label{eq:0234}
\text{there is some $z\in Q$ with $\mu(z)=d$.}
\end{equation}

In particular,~$Q$ is non-empty. Now 
Claim~\ref{cl:interactionconglob1}\eqref{it:upperlayerconglo} informs 
us that the nonnegative integer~$s$ defined by 
\[
s+1=|Q|=|K^{(d)}\sm \fB_1|
\]
can be bounded by $s\le (d+1)-(d-m+2)-1=m-2$. 

Setting
\begin{align*}
r'=&\max\{a\in\NN_0\colon a\leq s\text{ and }2^a\leq m-1\}\,,\\
r=&\max\{a\in\NN_0\colon a\leq s+1\text{ and }2^a\leq m\}\,,
\end{align*}
we contend that 
\begin{equation}\label{eq:1537}
k(x)\geq
\begin{cases}
	\frac{2^{d+1}-2}{d+1}-\big(\frac{s}{d}+\frac{m-2-s}{d+1-r'}\big)
	& \text{ if } x\in Q \cr
	\frac{2^{d+1}-2}{d+1}-\big(\frac{s+1}{d}+\frac{m-2-s}{d+1-r}-\frac{1}{2}\big)
	& \text{ if } x\in K\sm Q.
\end{cases}
\end{equation}

To see this, we first observe that if some $d$-set $D\in\fB_1$ 
satisfies $x\in D\not\subseteq K$, then Claim~\ref{cl:largemtnedgeconglo}
provides a stronger bound. So we may assume that such $d$-sets do not exist. 
Next, for $x\in K\sm Q$ we have $\mu(x)\ge d+1$ by the above discussion.
Moreover, due to the definition of $Q$, all but one of the $d-s$ sets 
in $K^{(d)}\cap \fB_1$ contain $x$ and, therefore, 
Claim~\ref{cl:bigandsmallmu}\eqref{it:bigmukspecial} leads to the desired estimate.
If $x\in Q$ and $\mu(x)=d$, we can argue 
similarly, but appealing to Claim~\ref{cl:bigandsmallmu}\eqref{it:smallmuk}
instead.
Lastly, if~$x\in Q$ and~$\mu(x)\ge d+1$, then~\eqref{eq:0234} gives~$s\ge 1$, and there are still $d-s$ sets in $K^{(d)}\cap \fB_1$ containing $x$.
By Claim~\ref{cl:bigandsmallmu}\eqref{it:bigmukspecial} applied to $s-1$ in place of $s$, we obtain
\[
k(x)
\ge
\frac{2^{d+1}-2}{d+1}-\Big(\frac{s}{d}+\frac{m-1-s}{d+1-r}-\frac{1}{2}\Big)\,.
\]
Combined with the first part of Fact~\ref{fa:dmqrr's} this concludes the 
proof of~\eqref{eq:1537}.

Lastly, the second part of Fact~\ref{fa:dmqrr's} yields
\begin{align*}
\sum_{x\in K}k(x)
&\geq 
(2^{d+1}-2)-
(s+1)\Big(\frac{s}{d}+\frac{m-2-s}{d+1-r'}\Big) \\
&\;\;\;\;\;-(d-s)\Big(\frac{s+1}{d}+\frac{m-2-s}{d+1-r}-\frac{1}{2}\Big) \\
&\geq 
2^{d+1}-m\,. \qedhere
\end{align*}
\end{proof}

If we could show that the conglomerates are mutually disjoint, then the proof 
of Theorem~\ref{thm:meta} could easily be completed as in~\S\ref{subsec:5}.
The route followed in the remaining pages of this section, however, is slightly 
different. Specifically, we shall only show that distinct conglomerates cannot  
intersect in two or more vertices (cf.\ Claim~\ref{cl:conglointersection} below).
It will then turn out that a slightly more subtle version of our argument can 
still be pushed through. Let us start with a special case.  

\begin{claim}\label{clm:1802}
If $d=4$, then any two conglomerates intersect in at most one vertex. 
\end{claim}

\begin{proof}
Assume that $K$, $K'$ are two distinct conglomerates whose 
intersection $B=K\cap K'$ has size at least~$2$. Due to 	
Claim~\ref{cl:interactionconglob1}\eqref{it:upperlayerconglo}
and Claim~\ref{cl:largemtnedgeconglo}, we know that~$|B|=2$. 
The latter statement also implies that there is no 
quadruple $D\in \fB_1\cap V^{(4)}$ with $B\subseteq D$. 
Yet Claim~\ref{cl:interactionconglob1}\eqref{it:b1closetocongloeverywhere} 
applied to either vertex $z\in B$ yields a 
set~$M\in\fB_1$ with~$B\subsetneq M\subseteq K$. The only possibility avoiding 
an immediate contradiction is that~$|M|=3$ and, moreover, $M$ is inclusion-maximal 
in $\fB_1$.

Let $x\in B$ be arbitrary and recall that Claim~\ref{cl:smalldegvtsexist} 
implies~$d_{\fB_1}(x)\le 2^4-m<16$.
However, Claim~\ref{cl:interactionconglob1}\eqref{it:b1closetocongloeverywhere} 
states that the number of sets~$T\in\fB_1$ with~$x\in T\subseteq K$ is at least~$9$ 
and that there are at least~$9$ sets~$T\in\fB_1$ with~$x\in T\subseteq K'$.
The only set counted twice is~$B$ and thus,~$d_{\fB_1}(x)\geq 17$, which is absurd.
\end{proof}

We prepare the general case by the following degree estimate. 

\begin{claim}
Suppose $d\ge 5$ and that $K$, $K'$ are distinct conglomerates. If a  
set $A\subseteq K\cap K'$ satisfies $\max\{2, |K\cap K'|-1\}\le |A|\le d-2$, 	
then the mountain
\[
\fB'_1=\fB_1\sm \{L\in\fB_1\colon A\subseteq L\text{ and } |L\sm A|\ge 2\}
\]
has the property that $d_{\fB'_1}(x)\ge 2^d-m$ holds for all $x\in A$.
\end{claim}

\begin{proof}
It is plain that $\fB'_1$ is indeed a mountain. Fix an arbitrary vertex~$x\in A$.
Claim~\ref{cl:interactionconglob1}\eqref{it:b1closetocongloeverywhere} 
yields integers~$\delta,\delta'\in[0,2m-1]$ such that
\begin{align*}
\vert\{L\in\fB_1\colon x\in L\subseteq K\}\vert=2^d-\delta 
\quad \text{ and } \quad 
\vert\{L\in\fB_1\colon x\in L\subseteq K'\}\vert=2^d-\delta'\,.
\end{align*}
Due to~$|A|\ge |K\cap K'|-1$, we have
\begin{align*}
&\phantom{\ge\ge}
\vert\{L\in\fB_1\colon x\in L\text{ and }(L\subseteq K\text{ or }L\subseteq K'\})\vert\\
&\ge
(2^d-\delta)+(2^d-\delta')-2^{|K\cap K'|-1}\\
&\ge
2^{d+1}-(\delta+\delta')-2^{|A|}\,.
\end{align*}
Let~$\eps,\eps'\geq 0$ be defined by
\[
|\{F\subseteq K\sm A\colon |F|\ge 2, A\cup F\in\fB_1\}|
= 
2^{d+1-|A|}-(d+2-|A|)-\eps
\]
and
\[
|\{F\subseteq K'\sm A\colon |F|\ge 2, A\cup F\in\fB_1\}|
=
2^{d+1-|A|}-(d+2-|A|)-\eps'\,.
\]
Then we can bound the degree of~$x$ in~$\fB_1'$ by
\begin{align}\label{eq:degreeinb1'lowerbound}
d_{\fB_1'}(x)
\ge 
2^{d+1}-\big(2^{|A|}+2^{d+2-|A|}\big)-(\delta+\delta')+(\eps+\eps')
+2(d+2-|A|)\,.
\end{align}

\smallskip

{\hskip2em \it First case: $|A|\ge 3$}

\smallskip

Since~$|A|\le d-2$, we get by the first observation in the proof of 
Fact~\ref{fa:ineq1} that
\[
2^{|A|-3}+2^{d-1-|A|}\le 1+ 2^{d-4}
\]
and thus,
\[
2^{|A|}+2^{d+2-|A|}\le 8+2^{d-1}\,.
\]
Using this as well as~$\delta+\delta'\le 4m-2$ and~$\eps+\eps'\ge 0$, 
we obtain
\begin{align*}
d_{\fB_1'}(x)
&\ge
2^{d+1}-(2^{d-1}+8)-(4m-2)+8
=
2^d+2^{d-1}+2-4m \\
&\ge
2^d-m+1\,,
\end{align*}
which is stronger than claimed.

\smallskip

{\hskip2em \it Second case: $|A|=2$}

\smallskip

Recall that~$\delta$ counts the elements in the 
set~$\cQ=\{L\colon x\in L\subseteq K\text{ and }L\not\in\fB_1\}$.
Consider the partition
\[
\cQ=\cQ_1\dcup\cQ_2\dcup\cQ_3
\] 
defined by
\begin{align*}
\cQ_1=&\{L\in\cQ\colon A\subseteq L\text{ and }\vert L\vert\geq 4\}\,,\\
\cQ_2=&\{L\in\cQ\colon A\subseteq L\text{ and }\vert L\vert\le 3\}\,,\\
\quad \cQ_3=&\{L\in\cQ\colon A\not\subseteq L\}\,,
\end{align*}
and note that~$\eps=\vert\cQ_1\vert$.
Claim~\ref{cl:interactionconglob1}\eqref{it:1415} entails $|L|=3$ for 
every~$L\in\cQ_2$ and~$|L|\ge 3$
for every~$L\in\cQ_3$. We would also like to point out that the sets 
in $\cQ_3$ contain~$x$ but not the other element of~$A$. 
Since~$\fB_1$ is a mountain, this shows 
that $L\longmapsto L\cup A$ is an injection from~$\cQ_3$ to~$\cQ_1$,
whence~$\vert\cQ_3\vert\leq\varepsilon$.

Now suppose for the sake of contradiction that~$\vert\cQ_2\vert\ge 2$. 
As every subset of~$K$ that contains one or two of 
the sets in~$\cQ_2$ must be missing in~$\fB_1$, this yields
\[
\vert\{F\colon x\in F\subseteq K\}\setminus\fB_1\vert
\geq 
3\cdot 2^{d-3}
>
2d-1
\ge
2m-1\,,
\]
contrary to Claim~\ref{cl:interactionconglob1}\eqref{it:b1closetocongloeverywhere}.

This proves $|\cQ_2|\le 1$, so that 
altogether~$\delta=\vert\cQ\vert\leq\eps+1+\eps$ 
and~$\delta-\eps\le\frac12(\delta+1)\le m$.
Analogously, we also have~$\delta'-\varepsilon'\leq m$.
Plugging this into~\eqref{eq:degreeinb1'lowerbound} we learn
\[
d_{\fB_1'}(x)
\ge 
2^d+2d-4-(\delta+\delta')+(\eps+\eps')
\ge 
2^d+(d-4)-m
\ge
2^d-m+1\,,
\]
as desired.
\end{proof}

We can now continue our analysis of conglomerate intersections. Incidentally, 
the proof of our next claim is the only place in the entire section where 
our freedom to create a new mountain and increase $t$ gets exploited. 
Accordingly, the argument that follows is the main reason why we switched
from discussing simplicial complexes to talking about mountains. 

\begin{claim}\label{cl:conglointersection}
Any two conglomerates intersect in at most one vertex.
\end{claim}

\begin{proof}
Assume for the sake of contradiction that~$K$ and~$K'$ are two conglomerates whose 
intersection~$B=K\cap K'$ satisfies $|B|\ge 2$. Owing to Claim~\ref{clm:1802}
we know $d\ge 5$ and, as there, 
Claim~\ref{cl:interactionconglob1}\eqref{it:upperlayerconglo}
combined with Claim~\ref{cl:largemtnedgeconglo} yields $|B|\le d-2$.

We contend that there exists a set $A\subseteq B$ such 
that $|A|\ge\max\{2, |B|-1\}$ and the set system		
\[
\fB_{t+1}
=
\{F\subseteq V\setminus A\colon\vert F\vert\geq 2\text{ and }A\cup F\in\fB_1\}
\]
is non-empty.	

In the special case~$|B|=2$, this can be seen by taking~$A=B$, for then
Claim~\ref{cl:interactionconglob1}\eqref{it:b1closetoconglo} yields
\begin{align*}
&\phantom{\le}
\vert\{F\subseteq K\sm B\colon |F|\ge 2 \text{ and } F\cup B\not\in\fB_1\}\vert
\le
\vert\{L\subseteq K\colon \vert L\vert\geq 2\text{ and }L\notin\fB_1\}\vert\\
&\leq
2m-2<2d-1<2^{d-1}-d 
=
\vert\{F\subseteq K\sm B\colon \vert F\vert\geq 2\}\vert\,,
\end{align*}
so that~$\fB_{t+1}$ is indeed non-empty.

If~$|B|>2$, we pick any~$D\in K^{(d)}\cap\fB_1$ (the existence of which is 
ensured by Claim~\ref{cl:interactionconglob1}\eqref{it:upperlayerconglo}), 
and set~$A=D\cap B$.
Clearly, we have~$\vert A\vert\geq\max\{2,\vert B\vert-1\}$.
Moreover, since~$\fB_1$ is a mountain and~$|B|\leq d-2$, the set system~$\fB_{t+1}$ 
contains~$D\sm A$ and is, in particular, non-empty.

Having thereby explained the choice of an appropriate set $A$, we stipulate 
\[
\fB_1'
=
\fB_1\sm \{L\colon A\subseteq L\text{ and }\vert L\sm A\vert\geq2\}\,.
\]
Clearly, both~$\fB_1'$ and~$\fB_{t+1}$ are mountains 
with~$\vert\fB_1\vert=\vert\fB_1'\vert+\vert\fB_{t+1}\vert$
and for all~$u\in V\setminus A$, 
we have~$d_{\fB_1'}(u)+d_{\fB_{t+1}}(u)=d_{\fB_1}(u)$.
Moreover, the previous claim gives~${d_{\fB_1'}(x)\ge 2^d-m}$ for all~$x\in A$.
So~$(\fB_1',\fB_2,\dots,\fB_{t+1}, f)$ satisfies all assumptions of 
the theorem. But~${\fB_{t+1}\ne\vn}$ yields $|\fB'_1|<|\fB_1|$ 
and we reach a contradiction to the minimality of~$\vert\fB_1\vert$.
\end{proof}

\begin{claim}\label{cl:ksurplusmulticonglos}
If a vertex~$x$ belongs to~$\lambda\geq 2$ conglomerates, then
\[
k(x)-\frac{2^{d+1}-m}{d+1}\geq \frac{\lambda(d-2)(d-3)}{4}\,.
\]
\end{claim}

\begin{proof}
Let~$K$ be any conglomerate containing~$x$.
Since~$\fB_1$ is a mountain, the set system
\[
\cA=\{\emptyset\}\cup\{F\setminus \{x\}\colon x\in F\in\fB_1
\text{ and }F\subseteq K\}
\]
is a simplicial complex.
Claim~\ref{cl:interactionconglob1}\eqref{it:b1closetocongloeverywhere}
tells us
\[
\vert\cA\vert\geq 1+2^d-(2m-1)
\geq 
2^d-2d+2\,.
\]
Now we can utilise Lemma~\ref{lem:weightboundsimple} to derive
\[
\sum_{F\in\cA}\frac{1}{\vert F\vert+1}
\geq\frac{2^{d+1}-1}{d+1}-\frac{2d-2}{3}\,,
\]
and thereby,
\[
\sum_{x\in F\subseteq K\text{ and }F\in\fB_1}\frac{1}{\vert F\vert}
\geq
\frac{2^{d+1}-1}{d+1}-\frac{2d-2}{3}-1\,.
\]

Due to Claim~\ref{cl:conglointersection}, any two conglomerates that 
contain~$x$ intersect only in~$x$, whence summing over all such 
conglomerates yields
\begin{align*}
k(x)
&\geq 
1+\lambda\Big(\frac{2^{d+1}-1}{d+1}-\frac{2d-2}{3}-1\Big)\\
&\geq 
1+\Big(1+\frac{\lambda}{2}\Big)
\Big(\frac{2^{d+1}-1}{d+1}-1\Big)-\lambda\frac{2d-2}{3}\,.
\end{align*}
This gives
\[
k(x)-\frac{2^{d+1}-1}{d+1}
\geq 
\lambda\Big(\frac{2^{d+1}-1}{2(d+1)}-\frac{1}{2}-\frac{2d-2}{3}\Big)
\]
and Fact~\ref{fa:ineq1+1/2} leads to the desired estimate. 
\end{proof}

For every vertex $x\in V$, let $\lambda(x)$ denote the number of conglomerates 
containing $x$. 

\begin{claim}\label{cl:normalisedconglomerateweight}
Every conglomerate $K$ satisfies
\[
\sum_{x\in K}\frac{k(x)}{\lambda(x)}
\geq
\frac{2^{d+1}-m}{d+1}\sum_{x\in K}\frac{1}{\lambda(x)}\,.
\]
\end{claim}

\begin{proof}
We consider the partition~$K=Q\dcup R\dcup U$ defined by
\begin{align*}
Q=&\{x\in K\colon \mu(x)=d\},\\
R=&\{x\in K\colon \mu(x)>d\text{ and }\lambda(x)=1\},\\
\text{ and } U=&\{x\in K\colon \mu(x)>d\text{ and }\lambda(x)>1\}\,.
\end{align*}

Summarising the lower bounds obtained in Claims~\ref{cl:bigandsmallmu}\eqref{it:smallmuk},~\ref{cl:bigandsmallmu}\eqref{it:bigmuk}, and~\ref{cl:ksurplusmulticonglos},
we have 
\[
k(x)-\frac{2^{d+1}-m}{d+1}\ge
\begin{cases}
-\frac{(d-2)(d-3)}{4d} & \text{ if } x\in Q \cr
0                      & \text{ if } x\in R \cr
\frac{\lambda(x)(d-2)(d-3)}{4} & \text{ if } x\in U.
\end{cases}
\]
Moreover, Claim~\ref{cl:weightheavyvts} discloses~$\lambda(x)=1$ for~$x\in Q$
and we can conclude 
\[
\sum_{x\in K}\frac{1}{\lambda(x)}\Big(k(x)-\frac{2^{d+1}-m}{d+1}\Big)
\geq
\frac{(d-2)(d-3)}{4}\Big(-\frac{\vert Q\vert}{d}+\vert U\vert\Big)\,.
\]
Provided that $|U|\ge 1$, this proves the claimed estimate. 
On the other hand, if $U=\vn$, then~$K$ is disjoint to all other 
conglomerates and we can invoke Claim~\ref{clm:1957}.
\end{proof}

Now let~$V'$ be the union of all conglomerates.
By summing Claim~\ref{cl:normalisedconglomerateweight} over all conglomerates, 
we obtain~$\sum_{x\in V'}k(x)\geq\frac{2^{d+1}-m}{d+1}\vert V'\vert$.
Vertices $x\in V\sm V'$ satisfy $\mu(x)>d$ (by Claim~\ref{cl:bigandsmallmu}\eqref{it:munottoosmall} and Claim~\ref{cl:weightheavyvts}) 
and, hence, $k(x)\ge \frac{2^{d+1}-m}{d+1}$ (by 
Claim~\ref{cl:bigandsmallmu}\eqref{it:bigmuk}). 
Summing everything we infer 
\[
\sum_{\tau=1}^t\vert\fB_{\tau}\vert+\sum_{x\in V}\Big(\frac{f(x)}{2}+1\Big)
=
\sum_{x\in V}k(x)
\ge	
\frac{2^{d+1}-m}{d+1}\vert V\vert\,,
\]
which concludes the proof of Theorem~\ref{thm:meta}.

\section{New regimes}\label{sec:newreg}
In this section we use the same approach to push beyond the regime~$m\leq d$.
First we deal with the special case~$m=d+1$ and~$d=4$, proving a conjecture by Frankl and Watanabe~\cite{FW:94}.

\begin{theorem}\label{thm:beyondspec}
For a set~$V$, let~$\fB_1,\dots,\fB_t\subseteq \powerset(V)$ be mountains, 
and let~${f,g,h\colon V\lra\NN_0}$ be functions
such that the following conditions hold.
\begin{enumerate}
\item\label{it:cond1} If~$F\in\fB_2\cup\dots\cup\fB_t$, 
then~$\vert F\vert\leq 3$.
\item\label{it:cond2} For all~$x\in V$, we have
\[
f(x)+7g(x)+10h(x)+\sum_{\tau=1}^t d_{\fB_{\tau}}(x)\geq 11\,.
\]
\end{enumerate}
Then 
\[
\sum_{\tau=1}^t\vert\fB_{\tau}\vert
+
\sum_{x\in V}\Bigl(1+\frac{f(x)}{2}+\frac{51g(x)}{20}+\frac{19h(x)}{5}\Bigr)
\geq
\frac{53}{10}\vert V\vert\,.
\]
\end{theorem}

\begin{proof}[Proof of Theorem~\ref{thm:11} assuming Theorem~\ref{thm:beyondspec}]
The upper bound is provided by Construction~\ref{constr:beyond}.
Given a simplicial complex~$\cS$ with vertex set $V$ and~$\delta(\cS)\geq 12$, Theorem~\ref{thm:beyondspec} applied 
to~$t=1$,~$\fB_1=\{F\in\cS\colon \vert F\vert\geq 2\}$, 
and three zero functions~$f, g, h\colon V\lra\{0\}$ yields 
\[
\vert\fB_1\vert+\vert V\vert\geq\frac{53}{10}\vert V\vert\,.
\]
Since~$\vert\cS\vert=\vert\fB_1\vert+\vert V\vert+1$, 
this proves that we have indeed $|\cS|>\frac{53}{10}|V|$.
\end{proof}

The remainder of this section is devoted to the proof of Theorem~\ref{thm:beyondspec}.
Consider a counterexample~$(\fB_1,\dots,\fB_t,V,f,g,h)$ which 
lexicographically minimises the ordered triple
\[
	\Bigl(\vert\fB_1\vert, \sum_{\tau\in[t]}\vert\fB_{\tau}\vert,
	\sum_{x\in V}\bigl(f(x)+g(x)+h(x)\bigr)\Bigr)\,.
\]
Let
\[
\Omega
=
\sum_{\tau=1}^t\vert\fB_{\tau}\vert
+
\sum_{x\in V}\Bigl(1+\frac{f(x)}{2}+\frac{51g(x)}{20}+\frac{19h(x)}{5}\Bigr)
\]
be the left side of the inequality we seek to establish. 
For every~$x\in V$, we set
\begin{align*}
\ds(x)
&=
f(x)+7g(x)+10h(x)+\sum_{\tau=1}^t d_{\fB_{\tau}}(x) \\
\text{and} \quad k(x)
&=
1+\frac{f(x)}{2}+\frac{51g(x)}{20}+\frac{19h(x)}{5}
+
\sum_{\tau=1}^t\sum_{x\in F\in\fB_{\tau}}\frac{1}{\vert F\vert}\,.
\end{align*}
So we are given that~$\ds(x)$ is always at least~$11$ and, 
due to~$\Omega=\sum_{x\in V}k(x)$, we need to show that~$k(x)$
is on average at least $\frac{53}{10}$. A straightforward adaptation of the proof 
of Claim~\ref{cl:smalldegvtsexist} discloses that 

\begin{equation}
\tag{$\star$}\label{eq:fmax}
\parbox{\dimexpr\linewidth-6em}{    \strut
\it
if $\tau\in [t]$ and $F\in\fB_\tau$ is inclusion-maximal, then there are 
at least three vertices $x\in F$ such that $\ds(x)=11$.
\strut
}
\end{equation}

In particular, we have $\vert F\vert\leq 4$ for all $F\in\fB_1$.	For brevity 
we will call the $4$-sets in~$\fB_1$ {\it quadruples}. We proceed by analysing
the possible intersections of quadruples. 

\begin{claim}\label{cl:4setsintsize}
Any two quadruples intersect in at most one vertex.
\end{claim}

\begin{proof}
First notice that if two quadruples intersect in exactly two vertices, 
say~$x$ and~$y$, then~$\ds(x)\ge d_{\fB_1}(x)\ge 13>11$, and, similarly, 
$\ds(y)>11$, which contradicts~\eqref{eq:fmax}.

Hence, we are left to find a contradiction in the case that two quadruples
intersect in three vertices. Assume that~$u, v, x, y, z\in V$ are such 
that~$uxyz, vxyz\in\fB_1$.
Due to~$d_{\fB_1}(x), d_{\fB_1}(y), d_{\fB_1}(z)\geq 11$
and~\eqref{eq:fmax}, we can suppose, without loss of generality, 
that~$\ds(x)=d_{\fB_1}(x)=11$ and~$\ds(y)=d_{\fB_1}(y)=11$.

\smallskip

{\it \hskip2em First case: $\ds(z)=11$.}

\smallskip

In this case, the only edges in~$\bigcup_{\tau\in [t]}\fB_\tau$ 
containing at least one of~$x$,~$y$, and~$z$ are those contained 
in~$uxyz$ or~$vxyz$. Let~$\cE$ be the set of these $18$ edges,
\begin{enumerate}
\item[$\bullet$] set $V'=V\sm \{x, y, z\}$, 
\item[$\bullet$] define~$g'\colon V'\lra\NN_0$ by 
setting~$g'(u)=g(u)+1$,~$g'(v)=g(v)+1$, and~$g'(w)=g(w)$ for 
all~$w\in V'\sm\{u, v\}$,
\item[$\bullet$] and let $f'$ and $h'$ be the restrictions of $f$ and $h$ to $V'$, respectively.
\end{enumerate}
Since all assumptions of our theorem hold for~$(\fB_1\sm\cE, \fB_2, \dots, \fB_t, V', f', g', h')$, the minimality of $|\fB_1|$ yields
\[
\Omega-18-3+2\cdot\frac{51}{20}
\ge
\frac{53}{10}(\vert V\vert-3)\,,
\]
which simplifies to~$\Omega\geq\frac{53}{10}\vert V\vert$.

\smallskip

{\it \hskip1em Second case: $\ds(z)\ge 12$.}

\smallskip

Due to the minimality of $\sum_{w\in V}(f(w)+g(w)+h(w))$, we have 
$f(z)=g(z)=h(z)=0$ and, thus, there are an index $\tau\in [t]$ and 
a set~$M\in \fB_\tau$ such that $z\in M$ but, if~$M\in\fB_1$, then neither~$M\subseteq uxyz$ nor~$M\subseteq vxyz$. Choosing~$M$ inclusion-maximal with these properties, 
we have $|M|\ge 4$ due to~\eqref{eq:fmax}. Therefore,~$\tau=1$ and~$M$
is a quadruple distinct from~$uxyz$ and~$vxyz$. Thus,
\[
\fB_{t+1}
=
\{F\sm\{z\}\colon z\in F\in\fB_1\text{ and } |F|\ge 3 \text{ and }x,y\not\in F\}
\] 
is a non-empty mountain. We contend that 
\[
\fB_1'
=
\fB_1\sm\{F\colon |F|\ge 3 \text{ and } z\in F \text{ and } x,y\not\in F\}\,,
\]
is a mountain as well. To see this, we consider any two
sets~$F$, $F'$ such that~$F\subsetneq F'\in\fB_1'$ and~$|F|\ge 2$.
If~$z\not\in F$ or~$\vert F\vert=2$, we immediately get~$F\in\fB_1'$.
So we may assume that~$z\in F$,~$\vert F\vert=3$, and~$\vert F'\vert=4$.
Then for~$F'$ to be in~$\fB_1'$, we need to have~$x\in F'$ or~$y\in F'$.
But the only quadruples containing~$x$ or~$y$ are~$uxyz$ and~$vxyz$. 
Hence,~$F$ contains~$x$ or~$y$, whereby~$F\in\fB_1'$.
Altogether,~$\fB_1'$ is indeed a mountain.

Since 
\[
(\fB_1',\fB_2,\dots,\fB_t,\fB_{t+1},V,f,g,h)
\]
also satisfies the condition of the theorem, $|\fB_1|=|\fB'_1|+|\fB_{t+1}|$
and the minimality of $|\fB_1|$ yield the desired 
estimate~$\Omega\ge \frac{53}{10}|V|$.
\end{proof}

Denote by~$\lambda(x)$ the number of quadruples containing a given vertex~$x$.

\begin{claim}\label{cl:smallweight}
Let~$x\in V$.
\begin{enumerate}
\item\label{it:631} If $\lambda(x)=0$, then $k(x)\ge\frac{53}{10}$.
\item\label{it:632} If $\lambda(x)=1$ and $k(x)<\frac{53}{10}+\frac{1}{12}$,
then~$k(x)=\frac{53}{10}-\frac{1}{20}$ and~$\ds(x)=d_{\fB_1}(x)=11$.
Moreover, there are distinct vertices~$y_1,\dots,y_4$ such 
that~$\fB_1$ contains all four edges~$xy_i$, all six edges~$xy_iy_j$, and the 
quadruple~$xy_1y_2y_3$.  
\item\label{it:633} If $\lambda(x)\ge 2$, 
then~$k(x)\ge\frac{53}{10}+\frac{3}{5}\lambda(x)$.
\end{enumerate}
\end{claim}

\begin{proof}
Starting with~\eqref{it:631} and~\eqref{it:632} we first suppose $\lambda(x)\le 1$.
The construction we saw right at the beginning of the proof 
of Claim~\ref{cl:bigandsmallmu} yields a simplicial complex $\cA$ with 
\begin{align*}
|\cA|&=1+f(x)+\sum_{\tau=1}^t d_{\fB_\tau}(x)\ge 12-7g(x)-10h(x)\,, \\
k(x)&=\sum_{A\in \cA}\frac 1{|A|+1}+\frac{51g(x)}{20}+\frac{19h(x)}{5}\,, \\
\text{ and } \quad 
\lambda(x)&=|V(\cA)^{(3)}\cap \cA|\,.
\end{align*}
Let us observe that
\begin{center}
\begin{tabular}{c|c}
if & then \\ \hline
$g(x)+h(x)\ge 2$ & $k(x)\ge 1+\frac52(g(x)+h(x))\ge 6$ \\
$g(x)=0$ and $h(x)=1$ & $|\cA|\ge 2$ 
and $k(x)\ge 1+\frac 12+\frac{19}5+\frac{\lambda(x)}{4}=\frac{53}{10}
+\frac{\lambda(x)}{4}$ \\
$g(x)=1$ and~$h(x)=0$ & $|\cA|\ge 5$, whence 
$k(x)\ge 1+\frac 32+\frac 13+\frac{51}{20}=\frac{53}{10}+\frac 1{12}$.
\end{tabular}
\end{center}

It remains to discuss the case that $g(x)=h(x)=0$ and $|\cA|\ge 12$.
Setting $\mu=|V(\cA)|$ we have $|V(\cA)^{(2)}\cap \cA|\ge 11-\mu-\lambda(x)$
and, therefore, 
\[
k(x)
\ge 
1+\frac\mu 2+\frac{11-\mu-\lambda(x)}3+\frac{\lambda(x)}4 
=
\frac{14}3+\frac\mu 6-\frac{\lambda(x)}{12}\,.
\]
If $\mu\ge 5$, this yields 
$k(x)\ge \frac{14}3+\frac 56-\frac{1}{12}>\frac{53}{10}+\frac{1}{12}$
and we are done. So we can assume~$\mu\le 4$ and 
$6\le 7-\lambda(x)\le |V(\cA)^{(2)}\cap \cA| \le \binom \mu 2$, whence $\mu=4$ 
and~$\lambda(x)=1$. This means that~$\cA$ has four vertices, all six 
possible $2$-sets, and one $3$-set. As~$\cA$ cannot be expressed in a non-trivial 
way as a union of two smaller vertex-disjoint simplicial complexes, all its non-empty edges come 
from the same mountain. By assumption~\eqref{it:cond1} of the theorem, this mountain 
can only be~$\fB_1$ and the `neighbourhood' of $x$ looks as described in 
clause~\eqref{it:632}. 
Finally, we have $k(x)=\frac{14}3+\frac23-\frac 1{12}=\frac{53}{10}-\frac{1}{20}$,
which completes the proof of~\eqref{it:631} and~\eqref{it:632}.  

Turning to~\eqref{it:633} we observe that every quadruple~$Q$ 
with~$x\in Q$ satisfies
\[
\sum_{x\in A\subseteq Q}\frac{1}{\vert A\vert}
=
\frac{15}{4}\,.
\]
Owing to Claim~\ref{cl:4setsintsize}, this yields 
\[
k(x)
\ge
1+\sum_{x\in A\in\fB_1}\frac{1}{\vert A\vert}
\geq 
1+\Bigl(\frac{15}{4}-1\Bigr)\lambda(x)
\geq 
1+\frac{43}{20}\lambda(x)+\frac{12}{20}\lambda(x)
\geq
\frac{53}{10}+\frac{3}{5}\lambda(x)\,.\qedhere
\]
\end{proof}

\begin{claim}\label{clm:nega}
Every quadruple $Q$ satisfies 
\[
\sum_{x\in Q}\frac{k(x)}{\lambda(x)}
\geq
\frac{53}{10}\sum_{x\in Q}\frac{1}{\lambda(x)}\,.
\]
\end{claim}

\begin{proof}
Assume for the sake of contradiction that 
\begin{equation}\label{eq:1542}
\sum_{x\in Q}\frac{1}{\lambda(x)}\Bigl(k(x)-\frac{53}{10}\Bigr)<0\,.
\end{equation}

By Claim~\ref{cl:smallweight} each of the 
differences $k(x)-\frac{53}{10}$ is at least $-\frac{1}{20}$ and, therefore, 
we have $k(x)<\frac{53}{10}+\frac{3}{20}\lambda(x)$ for every $x\in Q$. 
Because of Claim~\ref{cl:smallweight}\eqref{it:633}, this is only possible 
if $\lambda(x)=1$, i.e., if $Q$ is disjoint to all other quadruples. 
Now~\eqref{eq:1542} simplifies to 
\begin{equation}\label{eq:1721}
\sum_{x\in Q}\Bigl(k(x)-\frac{53}{10}\Bigr)<0\,.
\end{equation}
Appealing again to Claim~\ref{cl:smallweight}\eqref{it:632}, we learn that each of 
these four differences is either equal to $-\frac 1{20}$ or at least $\frac 1{12}$.
Consequently, we can enumerate $Q=\{x_1,\dots,x_4\}$ in such a way that
$k(x_1)=k(x_2)=k(x_3)=\frac{53}{10}-\frac 1{20}$. 
Once more by Claim~\ref{cl:smallweight}\eqref{it:632} this leads to a vertex~$y$ 
such that all $2$-sets~$x_iy$ with~$i\in [4]$ and all $3$-sets~$x_ix_jy$ 
with~$ij\in [4]^{(2)}$ are in~$\fB_1$. 
In particular, we thereby know four $2$-sets, six $3$-sets, and one quadruple 
containing~$x_4$. Thus we have found the 
terms $1+\frac42+\frac63+\frac14=\frac{53}{10}-\frac 1{20}$ contributing 
to $k(x_4)$. Any further contribution to~$k(x_4)$, if there exists any, 
needs to be at least $\frac 14$ and would therefore give 
$\sum_{i=1}^4(k(x_i)-\frac{53}{10})\ge \frac14-\frac 4{20}>0$, contrary 
to~\eqref{eq:1721}.

It remains to discuss the case that all edges intersecting $Q$ are 
in~$\cE=\powerset(Q\cup\{y\})\cap \fB_1$ and none of them is in~$\bigcup_{\tau\in[2,t]}\fB_{\tau}$. 
We observe that $|\cE|=21$, $d_\cE(y)=10$, 
\begin{enumerate}
\item[$\bullet$] set $V'=V\sm Q$,  
\item[$\bullet$] define~$h'\colon V'\lra\NN_0$ by~$h'(y)=h(y)+1$ 
and~$h'(z)=h(z)$ for all~$z\in V'\sm\{y\}$,
\item[$\bullet$] and let $f'$ and $g'$ be the restrictions of $f$ and $g$ to $V'$, respectively.
\end{enumerate}
Since~$(\fB_1\sm\cE,\fB_2,\dots,\fB_t,V',f',g',h')$
satisfies all assumptions of our theorem, the minimality of $\fB_1$
discloses 
\[
\Omega-21-4+\frac{19}{5}
\ge
\frac{53}{10}(\vert V\vert-4)\,,
\]
whence~$\Omega\geq\frac{53}{10}\vert V\vert$. This contradiction concludes the 
proof of Claim~\ref{clm:nega}. 
\end{proof}

Summing this result over all quadruples, we learn that the union~$V'$ of all
quadruples satisfies $\sum_{x\in V'} k(x)\ge\frac{53}{10}|V'|$. Together with 
the fact that $k(x)\ge \frac{53}{10}$ holds for all $x\in V\sm V'$ 
(cf.\ Claim~\ref{cl:smallweight}\eqref{it:631}), this 
proves~$\Omega=\sum_{x\in V} k(x)\ge\frac{53}{10}|V|$. 
Thereby Theorem~\ref{thm:beyondspec} is proved.

\end{document}